\documentclass[11pt]{article}
\usepackage{amsfonts}
\usepackage{color}
\usepackage{amsmath,amssymb,comment}
\newtheorem{theo}{Theorem}[section]
\newtheorem{lem}[theo]{Lemma}
\newtheorem{cor}[theo]{Corollary}

\newtheorem{defi}[theo]{Definition}

\newcommand{\mysection}[1]{\section{#1} \setcounter{equation}{0}}
\newcommand{\proof}{{\sc Proof.} \quad}
\newcommand{\proofc}{{\sc Proof} \ }
\newcommand{\be}{\begin{equation} \label}
\newcommand{\ee}{\end{equation}}
\newcommand{\bea}{\begin{eqnarray}\label}
\newcommand{\eea}{\end{eqnarray}}
\newcommand{\bas}{\begin{eqnarray*}}
\newcommand{\eas}{\end{eqnarray*}}
\newcommand{\bit}{\begin{itemize}}
\newcommand{\eit}{\end{itemize}}
\newcommand{\qed}{\hfill$\Box$ \vskip.2cm}
\newcommand{\nn}{\nonumber}
\newcommand{\R}{\mathbb{R}}
\newcommand{\N}{\mathbb{N}}
\newcommand{\pO}{\partial\Omega}

\newcommand{\eps}{\varepsilon}

\newcommand{\supp}{{\rm supp} \, }

\newcommand{\wto}{\rightharpoonup}
\newcommand{\wsto}{\stackrel{\star}{\rightharpoonup}}
\newcommand{\hra}{\hookrightarrow}
\newcommand{\io}{\int_\Omega}

\newcommand{\Del}{\Delta}
\newcommand{\del}{\delta}
\newcommand{\al}{\alpha}
\newcommand{\vt}{\vartheta}

\newcommand{\pa}{\partial}
\newcommand{\bom}{\overline{\Omega}}
\newcommand{\Om}{\Omega}

\newcommand{\wh}{\widehat}
\newcommand{\wt}{\widetilde}
\newcommand{\vs}{\vspace*}
\newcommand{\hs}{\hspace*}

\newcommand{\vp}{\varphi}
\newcommand{\lbal}{\left\{ \begin{array}{l}}
\newcommand{\lball}{\left\{ \begin{array}{ll}}
\newcommand{\ear}{\end{array} \right.}

\definecolor{orange}{rgb}{1,0.45,0}

\newcommand{\abs}{\\[5pt]}

\newcommand{\adb}{\allowdisplaybreaks}
\newcommand{\tm}{T_{max}}
\newcommand{\tme}{T_{max,\eps}}
\newcommand{\ts}{t_\star}

\newcommand{\ueps}{u_\eps}
\newcommand{\veps}{v_\eps}
\newcommand{\weps}{w_\eps}

\newcommand{\yeps}{y_\eps}

\newcommand{\Teps}{\Theta_\eps}
\newcommand{\gaeps}{\gamma_\eps}
\newcommand{\hgeps}{\wh{\gamma}_\eps}
\newcommand{\Gaeps}{\Gamma_\eps}
\newcommand{\wepsx}{w_{\eps x}}
\newcommand{\wepsxx}{w_{\eps xx}}

\newcommand{\wepst}{w_{\eps t}}
\newcommand{\vepsx}{v_{\eps x}}
\newcommand{\vepsxx}{v_{\eps xx}}
\newcommand{\vepsxxx}{v_{\eps xxx}}
\newcommand{\vepsxxxx}{v_{\eps xxxx}}
\newcommand{\vepst}{v_{\eps t}}
\newcommand{\vepsxt}{v_{\eps xt}}
\newcommand{\vepsxxt}{v_{\eps xxt}}
\newcommand{\uepsx}{u_{\eps x}}
\newcommand{\uepsxx}{u_{\eps xx}}
\newcommand{\uepsxxx}{u_{\eps xxx}}
\newcommand{\uepsxxxx}{u_{\eps xxxx}}
\newcommand{\uepsxxxxx}{u_{\eps xxxxx}}
\newcommand{\uepst}{u_{\eps t}}

\newcommand{\uepsxxt}{u_{\eps xxt}}
\newcommand{\uepsxxxt}{u_{\eps xxxt}}
\newcommand{\Tepsx}{\Theta_{\eps x}}
\newcommand{\Tepsxx}{\Theta_{\eps xx}}
\newcommand{\Tepst}{\Theta_{\eps t}}
\newcommand{\Tepstt}{\Theta_{\eps tt}}

\newcommand{\hg}{\wh{\gamma}}
\newcommand{\epss}{\eps_\star}
\newcommand{\epsss}{\eps_{\star\star}}

\newcommand{\whu}{\wh{u}}
\newcommand{\wtu}{\widetilde{u}}
\newcommand{\wtt}{\widetilde{\Theta}}
\newcommand{\Ts}{T_\star}
\begin{document}
\adb
\title{Local strong solutions in a quasilinear Moore-Gibson-Thompson type model
for thermoviscoelastic evolution in a standard linear solid}
\author{
Leander Claes\footnote{claes@emt.uni-paderborn.de}\\
{\small Universit\"at Paderborn}\\
{\small Institut f\"ur Elektrotechnik und Informationstechnik}\\
{\small 33098 Paderborn, Germany}
\and
Michael Winkler\footnote{michael.winkler@math.uni-paderborn.de}\\
{\small Universit\"at Paderborn}\\
{\small Institut f\"ur Mathematik}\\
{\small 33098 Paderborn, Germany}}
\date{}
\maketitle
\begin{abstract}
\noindent
This manuscript is concerned with the evolution system
\bas
	\lbal
	u_{ttt} + \al u_{tt} = \big(\gamma(\Theta) u_{xt}\big)_x + \big( \hg(\Theta) u_x\big)_x, \\[1mm]
	\Theta_t = D \Theta_{xx} + \Gamma(\Theta) u_{xt}^2,
	\ear
\eas
which arises as a simplified model for heat generation during acoustic wave propagation in a one-dimensional viscoelastic medium of
standard linear solid type.\abs
Under the assumptions that $D>0$ and $\al\ge 0$, and that $\gamma, \hg$ and $\Gamma$ are sufficiently smooth
with $\gamma>0, \hg>0$ and $\Gamma\ge 0$ on $[0,\infty)$, for suitably regular initial data a statement on local
existence and uniqueness of solutions in an associated Neumann problem is derived in a suitable framework of strong solvability.\abs
\end{abstract}
%
%
%
%
%35D30 Weak solutions to PDEs
%35L05 Wave equation
%74H20 Existence of solutions of dynamical problems in solid mechanics
%74F05 Thermal effects in solid mechanics
%
%
%
\newpage
\section{Introduction}\label{intro}
The mathematical analysis of models for the generation of heat during acoustic wave propagation has considerably thriven
during the past few decades.
Beyond the classical systems of thermoelasticity (\cite{slemrod}, \cite{racke90}, \cite{racke_shibata})
and thermoviscoelasticity
(\cite{dafermos}, \cite{roubicek}, \cite{blanchard_guibe}, \cite{racke_zheng_JDE1997}) that have undergone thorough
investigation already in the 1980s and 1990s,
more recently the literature has also addressed models in which the evolution equations governing the
corresponding mechanical parts lose the character of classical wave equations either by involving certain solution-dependent
degeneracies, or by containing third-order time derivatives.
Typical examples of the former type include second-order equations of Westervelt type,
for which due to the presence of quadratic nonlinearities a comprehensive solution theory seems to have been established
only in settings of suitably small initial data
(\cite{nikolic_said_JDE2022}, \cite{nikolic_said_NON2022}, \cite{wilke}, \cite{benabbas_said_SIMA2024}),
even in the absence of couplings to temperature fields (\cite{kaltenbacher_lasiecka_DCDSS}, \cite{meyer_wilke_AMOP},
\cite{doerfler});
similar observations apply to the related class of Kuznetsov equations (\cite{mizohata_ukai}, \cite{dekkers}).\abs
Acoustic models of Moore-Gibson-Thompson type (\cite{moore_gibson}, \cite{thompson})
go along with significantly increased challenges already due to the mere
inclusion of third-order time derivatives, as forming one of their core characteristics.
Accordingly, dichotomies between stability and instability newly arise even at levels of linear models,
which in the context of the prototype
\be{MGT}
	\tau u_{ttt} + \al u_{tt} = \hg \Del u_t + \gamma \Del u,
\ee
with constant positive parameters $\tau,\al,\hg$ and $\gamma$,
become manifest in statements on
large time decay under the condition $\al\hg>\tau\gamma$, and on the occurrence of
infinite-time grow-up when $\al\hg<\tau\gamma$ (\cite{kal_las_marchand2011}, \cite{marchand_mcdevitt_triggiani2012},
\cite{delloro_pata_AMOP2017}, \cite{chen_ikehata}; cf.~also
\cite{chen2025} and \cite{chen_ikehata} for associated Cauchy problems in $\R^n$, as well as
\cite{lasiecka_wang}, \cite{lasiecka_wang2}, \cite{delloro_las_pata2016} and \cite{alves} for variants involving memory terms).\abs
The analysis of nonlinear Moore-Gibson-Thompson models so far seems to have mainly concentrated on semilinear cases
including various types of forcing terms which either depend on the solution itself or on derivatives thereof.
A prominent representative is the Jordan-Moore-Gibson-Thompson equation
\be{JMGT}
	\tau u_{ttt} + \al u_{tt} = \hg \Del u_t + \gamma \Del u + k (u^2)_{tt},
\ee
which is known to admit global solutions for suitably small initial data,
either in bounded domains within the dissipative regime $\al\hg>\tau\gamma$
(\cite{kaltenbacher_lasiecka_pospieszalska}), or when posed as a Cauchy problem in the whole space
(\cite{racke_saidhouari}, \cite{said_JMGT});
on the other hand, a recent result asserts finite-time blow-up of some solutions
with respect to the norm of $u$ in $L^\infty$ (\cite{nikolic_win}).
Studies addressing relatives involving different alternative types of nonlinearities, partially depending on solutions gradients,
can be found documented with regard to global existence of small-data solutions in \cite{kaltenbacher_nikolic_M3AS2019},
and with respect to findings on nonexistence as well as blow-up in
\cite{chen2022}, \cite{chen2020}, \cite{chen_ikehata}, \cite{ming} and \cite{chen2021}, for instance.\abs
{\bf A quasilinear system involving Moore-Gibson-Thompson dynamics with temperature-dependent ingredients.} \quad
The present manuscript now focuses on a class of models which may be viewed as extensions of (\ref{MGT}) to situations
in which elastic parameters may depend on an external variable that forms an additional unknown.
Such types of couplings have been studied in contexts of Westervelt-Pennes systems in which
the mechanical part, describing an acoustic pressure $u$,
can be obtained by letting $\tau=0$ and considering $\gamma$ and $k$ as functions of the unknown
temperature that obeys a suitably forced heat equation (\cite{nikolic_said_JDE2022}, \cite{nikolic_said_NON2022},
\cite{benabbas_said_ZAMP2025}, \cite{benabbas_said_SIMA2024}, \cite{careaga_nikolic_said_JNLS2025});
further quasilinear examples, similarly including time derivatives of up to second order only,
arise in the modelling of acoustic wave propagation interacting with temperature fields in
solid materials of Kelvin-Voigt type and with temperature-dependent elastic parameters
(\cite{claes_lankeit_win}, \cite{claes_win}, \cite{fricke}, \cite{meyer}, \cite{win_AMOP}, \cite{win_SIMA}).
Studies on models that couple heat dynamics to Moore-Gibson-Thompson type equations involving third-order time derivatives
such as (\ref{MGT}), however, seem to have predominantly focused
on linear systems so far (see \cite{alves2013}, \cite{delloro_pata}, \cite{wang_liu_racke} and \cite{adachi} for some examples).
An apparently single exception has recently been concerned with the so-called Moore-Gibson-Thompson-Pennes system in which
the coefficients $\hg$, $\gamma$ and $k$ in (\ref{JMGT}) are allowed to exhibit some variations with respect to a temperature variable
the evolution of which in turn is governed by a heat equation that contains sources depending on $u_t$;
a corresponding result on local existence of solutions emanating from suitably small initial data has been derived in
\cite{benabbas_said_NONRWA}.\abs
To describe our specific object of study and its application background, let us consider material in which the viscoelastic properties are described by a Zener-type model (\cite{GutierrezLemini2014}).
The relation between mechanical stress \(T\) and strain \(S\) is described as
\be{zener}
	T + \tau_\mathrm{rel} T_t = c(\tau_\mathrm{ret} S_t + S) ,
\ee
with the relaxation and retardation time constants (\(\tau_\mathrm{rel}\) and \(\tau_\mathrm{ret}\)) and the stiffness \(c\).
Inserting the one-dimensional equation of motion, which links the mechanical stress to the displacement \(u\) via the density \(\rho \)
according to
\be{motion}
	T_x = \rho u_{tt}
\ee
and the constitutive equation for the mechanical strain in one-dimensional settings,
\be{strain}
	S = \rho u_x,
\ee
yields the Moore-Gibson-Thompson equation
\be{mgt_phy}
	\tau_\mathrm{rel} u_{ttt} + u_{tt} = \left(\frac{\tau_\mathrm{rel} c}{\rho} u_{xt}\right)_x + {\left(\frac{c}{\rho} u_x\right)}_x
\ee
for the displacement \(u\).
Assuming that \(c\) depends on the temperature field \(\Theta \), we choose the following parametrization for the differential equation for \(u\):
\be{mgt_1d}
	u_{ttt} + \al u_{tt} = {\left(\gamma(\Theta) u_{xt}\right)}_x + {\left(\hg(\Theta) u_x\right)}_x .
\ee
This implies that \(\gamma(\Theta) \propto \hg(\Theta)\), however, the assumption can be dropped if temperature-dependent viscoelastic properties are to be considered.\abs
Effects on the temperature field are considered in the form of a diffusion equation.
As the application focus here is on high-frequency periodic processes like acoustic wave propagation,
the influence of linear direct and inverse thermoelasticity may be neglected, as their reversible nature would not yield a net transfer of mechanical to thermal energy on a temporal average.
However, mechanical losses induced by the odd-numbered derivatives in the More-Gibson-Thompson equation (\ref{mgt_1d}) are irreversible and thus result in long-term temperature increases.
To quantify the converted energy, the work \(P\) done on the material is given as
\be{work}
	P = T \cdot S_t
\ee
(\cite{Boley2012}).
Inserting the material description provided by the Zener-type model (equation~\ref{zener}) yields the identity
\be{more_work}
	P = c S \cdot S_t + \tau_\mathrm{ret} c S_t^2 - \tau_\mathrm{rel} S_t \cdot T_t .
\ee
The first term in this expression (\(c S S_t\)) is the stored potential mechanical energy in the material and easily shown to be zero on average for harmonic processes.
Both remaining terms describe mechanical losses which, by conservation of energy, are converted into thermal energy.
Given the physically motivated condition that \(\tau_\mathrm{rel} < \tau_\mathrm{ret}\) (\cite{Boley2012}) and assuming that the elastic behavior is dominant as compared to mechanical losses we may approximate \(T_t \approx c S_t\) in (\ref{more_work}) and arrive at a single term for the mechanical loss density \(Q\), which is identical to the thermal source density
\be{thermal_work}
	Q = (\tau_\mathrm{ret} - \tau_\mathrm{rel}) c S_t^2 ,
\ee
which can now also be expressed in terms of the mechanical displacement \(u\)
\be{thermal_work_u}
	Q = (\tau_\mathrm{ret} - \tau_\mathrm{rel}) c u_{xt}^2
\ee
without requiring the analysis of an integro-differential equation.
Inserting the thermal source density into a one-dimensional diffusion equation for the temperature field \(\Theta \) yields
the heat equation
\be{thermals_phy}
	\Theta_t = D \Theta_{xx} + (\tau_\mathrm{ret} - \tau_\mathrm{rel}) c u_{xt}^2
\ee
with exclusively local sources, whence again assuming temperature-dependent elastic properties we end up with
\be{thermals}
	\Theta_t = D \Theta_{xx} + \Gamma(\Theta) u_{xt}^2 .
\ee
Supplementing (\ref{mgt_1d})-(\ref{thermals}) by reasonable initial and boundary conditions, we are thus led to
considering the problem
\be{0}
	\lball
	u_{ttt} + \al u_{tt} = \big(\gamma(\Theta) u_{xt}\big)_x + \big( \hg(\Theta) u_x\big)_x,
	\qquad & x\in\Om, \ t>0, \\[1mm]
	\Theta_t = D \Theta_{xx} + \Gamma(\Theta) u_{xt}^2,
	\qquad & x\in\Om, \ t>0, \\[1mm]
	u_x=0, \quad \Theta_x=0,
	\qquad & x\in\pO, \ t>0, \\[1mm]
	u(x,0)=u_0(x), \quad u_t(x,0)=u_{0t}(x), \quad u_{tt}(x,0)=u_{0tt}(x), \quad \Theta(x,0)=\Theta_0(x),
	\qquad & x\in\Om,
	\ear
\ee
in an open bounded interval $\Om\subset\R$,
where $D>0$ and $\al\ge 0$ are constant parameters, where $\gamma$, $\hg$ and $\Gamma$ are given functions on $[0,\infty)$,
and where $u_0, u_{0t}, u_{0tt}$ and $\Theta_0$ are prescribed initial distributions of the displacement variable $u=u(x,t)$,
its derivatives $u_t$ and $u_{tt}$, and the temperature field $\Theta=\Theta(x,t)$.
In physical systems, excitation terms in $u$ may be realised by electric fields in piezoelectric media, making
the considerations in this manuscript suitable for the analysis of thermo-piezoelectric phenomena.\abs
Especially in the presence of functions $\gamma,\hg$ and $\Gamma$ which are allowed to reflect experimentally
observed temperature dependencies of elastic parameters (\cite{friesen}) in the mathematically extreme sense of
being unbounded, we do not expect (\ref{0}) to possess global solutions;
for a rigorous detection of finite-time blow-up, driven by suitably rapid growth of $\Gamma=\gamma=\hg$,
in a corresponding Kelvin-Voigt type simplification of (\ref{0}) in which
the third-order contribution $u_{ttt}$ is neglected, we refer to \cite{win_SIMA}.\abs
In view of this caveat, the following main result of this manuscript appears essentially optimal in this regard,
by namely asserting local existence and uniqueness of solutions, along with a suitable extensibility criterion.
In formulating this and throughout the sequel, given a bounded interval $\Om\subset\R$ and $p\in (1,\infty]$ we let
$W^{2,p}_N(\Om):=\{\vp\in W^{2,p}(\Om) \ | \ \frac{\pa\vp}{\pa\nu}=0 \mbox{ on } \pO\}$.
\begin{theo}\label{theo16}
  Let $\Om\subset\R$ be an open bounded interval, let $D>0$ and $\al\ge 0$, and suppose that
  \be{g}
	\lbal
	\gamma\in C^2([0,\infty))
	\mbox{ is such that $\gamma>0$ on } [0,\infty), \\[1mm]
	\hg \in C^2([0,\infty))
	\mbox{ is such that $\hg>0$ on } [0,\infty)
	\qquad \mbox{and} \\[1mm]
	\Gamma\in C^1([0,\infty))
	\mbox{ is such that $\Gamma\ge 0$ on } [0,\infty).
	\ear
  \ee
  Then whenever
  \be{init}
	\lbal
	u_0\in W^{2,2}_N(\Om), \\[1mm]
	u_{0t} \in W^{2,2}_N(\Om), \\[1mm]
	u_{0tt} \in W^{1,2}(\Om)
	\qquad \mbox{and} \\[1mm]
	\Theta_0\in W^{2,\infty}_N(\Om)
	\mbox{ is such that $\Theta_0\ge 0$ in $\Om$,}
	\ear
  \ee
  there exists $\tm\in (0,\infty]$ as well as a unique pair $(u,\Theta)$ of functions
  \be{16.01}
	\lbal
	u \in C^0([0,\tm);W^{2,2}_N(\Om))
	\qquad \mbox{and} \\[1mm]
	\Theta\in C^0([0,\tm);C^1(\bom)) \cap C^{2,1}(\bom\times (0,\tm)) \cap W^{1,\infty}_{loc}(\bom\times [0,\tm))
	\ear
  \ee
  which are such that
  \be{16.02}
	u_t\in C^0([0,\tm);C^1(\bom)) \cap L^\infty_{loc}([0,\tm);W^{2,2}_N(\Om))
  \ee
  and
  \be{16.03}
	u_{tt} \in C^0(\bom\times [0,\tm)) \cap L^\infty_{loc}((0,\tm);W^{1,2}(\Om)),
  \ee
  that $\Theta\ge 0$ in $\Om\times (0,\tm)$, that $(u,\Theta)$ forms a strong solution of (\ref{0}) in $\Om\times (0,\tm)$
  in the sense of Definition \ref{dw} below, and that
  \be{ext}
	\mbox{if $\tm<\infty$, \quad then \quad}
	\limsup_{t\nearrow \tm} \Big\{
	\|u_t(\cdot,t)\|_{W^{2,2}(\Om)}
	+ \|u_{tt}(\cdot,t)\|_{W^{1,2}(\Om)}
	+ \|\Theta(\cdot,t)\|_{W^{2,\infty}(\Om)} \Big\} = \infty.
  \ee
\end{theo}
{\bf Remark.} \quad
ii) \ In contrast to the existence result for the related Moore-Gibson-Thompson-Pennes system recently obtained in
\cite{benabbas_said_NONRWA}, Theorem \ref{theo16} does not impose any smallness assumption on the initial data. \abs
ii) \ The restriction to one-dimensional scenarios here is motivated by the ambition to address thermoviscoelastic evolution
in standard linear solids, which in contexts of Zener-type modelling leads to genuinely two-component systems of the form
(\ref{mgt_1d})-(\ref{thermals}) indeed only in one-dimensional cases in which displacements can be described by using scalar
variables $u$.
From a mathematical perspective, however, we remark here that higher-dimensional analogues of (\ref{0}) for scalar-valued
unknowns $u$ and $\Theta$ could well be covered by the analysis subsequently developed, with main parts of necessary
modifications reducing to an increase of regularity requirements concerning $u_0, u_{0t}$ and $u_{0tt}$ in (\ref{init}).
\mysection{Strong solutions. Uniqueness}
To begin with, let us specify the notion of solution that will be referred to below.
\begin{defi}\label{dw}
  Let $D>0$ and $\al\ge 0$,
  suppose that $\gamma\in C^0([0,\infty))$, $\hg\in C^0([0,\infty))$ and $\gamma\in C^0([0,\infty))$, and let
  $u_0\in W^{2,2}_N(\Om), u_{0t}\in W^{2,2}_N(\Om)$, $u_{0tt} \in W^{1,2}(\Om)$ and $0\le \Theta_0\in C^0(\bom)$.
  Then given $T\in (0,\infty]$, we will call a pair $(u,\Theta)$ of functions
  \be{w1}
	\lbal
	u \in C^0([0,T);C^1(\bom)) \cap L^\infty_{loc}([0,T);W^{2,2}_N(\Om))
	\qquad \mbox{and} \\[1mm]
	\Theta\in C^0(\bom\times [0,T)) \cap C^{2,1}(\bom\times (0,T))
		\cap W^{1,\infty}_{loc}(\bom\times [0,T))
	\ear
  \ee
  a {\em strong solution} of (\ref{0}) in $\Om\times (0,T)$ if
  \be{w2}
	\lbal
	u_t \in C^0([0,T);C^1(\bom)) \cap L^\infty_{loc}([0,T);W^{2,2}_N(\Om))
	\qquad \mbox{and} \\[1mm]
	u_{tt} \in C^0(\bom\times [0,T)) \cap L^\infty_{loc}([0,T);W^{1,2}(\Om)),
	\ear
  \ee
  if $\Theta\ge 0$ in $\Om\times (0,T)$, if
  \be{w3}
	u(\cdot,0)=u_0,
	\quad
	u_t(\cdot,0)=u_{0t}
	\quad \mbox{and} \quad
	\Theta(\cdot,0)=\Theta_0
	\qquad \mbox{in } \Om,
  \ee
  if
  \bea{wu}
	- \int_0^T \io u_{tt} \vp_t
	- \io u_{0tt} \vp(\cdot,0)
	+ \al \int_0^T \io u_{tt} \vp
	= - \int_0^T \io \gamma(\Theta) u_{xt} \vp_x
	- \int_0^T \io \hg(\Theta) u_x\vp_x
  \eea
  for each $\vp\in C_0^\infty(\bom\times [0,T))$, and if in the classical pointwise sense we have
  $\Theta_t=D\Theta_{xx} + \Gamma(\Theta) u_{xt}^2$ in $\Om\times (0,T)$ and $\Theta_x=0$ on $\pO\times (0,T)$.
\end{defi}
{\bf Remark.} \quad
  i) \ If in the above setting we choose $(\zeta_\eta)_{\eta\in (0,T)} \subset C_0^\infty([0,T))$ in such a way that
  $\zeta_\eta\equiv 1$ on $[0,\frac{\eta}{2}]$, $\zeta_\eta' \le 0$ and $\supp \zeta_\eta \subset [0,\eta]$ for all $\eta\in (0,T)$,
  then letting $\vp(x,t):=\zeta_\eta(t)$ for $x\in\bom, t\in [0,T)$ and $\eta\in (0,T)$ in (\ref{wu}) shows that
  since $\sup_{t\in (0,\eta)} \|u_{tt}(\cdot,t)-u_{tt}(\cdot,0)\|_{L^\infty(\Om)} \to 0$ as $\eta\searrow 0$ by (\ref{w2}),
  in addition to (\ref{w3}) we also have
  \bas
	u_{tt}(\cdot,0)=u_{0tt}.
  \eas
  ii) \ If beyond mere continuity it is assumed that $\gamma$ and $\hg$ even belong to $C^1([0,\infty))$, then for each
  strong solution $(u,\Theta)$ on $\Om\times (0,T)$, (\ref{wu})
  implies validity of the identity
  $u_{ttt} + \al u_{tt}= \gamma(\Theta) u_{xxt} + \gamma'(\Theta) \Theta_x u_{xt} + \hg(\Theta) u_{xx} + \hg'(\Theta) \Theta_x u_x$
  in the distributional sense on $\Om\times (0,T)$. The regularity properties in (\ref{w1}) and (\ref{w2}), and especially
  the requirement on boundedness of $\Theta_x$ contained in (\ref{w2}), thus ensure that in this case,
  \bas
	u_{ttt} \in L^\infty_{loc}([0,T);L^2(\Om)).
  \eas
\vs{2mm}

In our first step toward Theorem \ref{theo16}, by means of a suitably designed testing procedure we can make sure that
in any fixed time interval, at most one such solution can exist.
\begin{lem}\label{lem15}
  Let $D>0$ and $\al\ge 0$, and assume (\ref{g}) and (\ref{init}). Then for each $T>0$, there exists at most one strong solution
  of (\ref{0}) in $\Om\times (0,T)$ in the sense of Definition \ref{dw}.
\end{lem}
\proof
  Assuming $(u,\Theta)$ and $(\wtu,\wtt)$ to be two strong solutions of (\ref{0}) in $\Om\times (0,T)$, we fix an arbitrary
  $T_0\in (0,T)$ and let
  \be{15.1}
	d(x,t):=\lball
	u(x,t)-\wtu(x,t),
	\qquad & x\in\bom, \ t\in [0,T_0), \\[1mm]
	0,
	\qquad & x\in\bom, \ t<0,
	\ear
  \ee
  as well as
  \be{15.2}
	\del(x,t):=\Theta(x,t)-\wtt(x,t),
	\qquad x\in\bom, \ t\in [0,T_0),
  \ee
  and note that according to Definition \ref{dw}, all the functions $d,d_t,d_{tt}$ and $\del$ are continuous on their respective
  domains of definition, and that not only, as a trivial consequence, $d_{tt}(x,0)=0$ for all $x\in\bom$, but that moreover also
  $\del(x,0)=0$ for all $x\in\Om$.
  In particular, two applications of (\ref{wu}) show that for each $\vp\in C_0^\infty(\bom\times [0,T_0))$,
  \bea{15.3}
	- \int_0^{T_0} \io d_{tt} \vp_t + \al \int_0^{T_0} \io d_{tt} \vp
	&=& - \int_0^{T_0} \io  \gamma(\Theta) u_{xt} \vp_x
	+ \int_0^{T_0} \io \gamma(\wtt) \wtu_{xt} \vp_x \nn\\
	& & - \int_0^{T_0} \io \hg(\Theta) u_x \vp_x
	+ \int_0^{T_0} \io \hg(\wtt) \wtu_x \vp_x \nn\\
	&=& - \int_0^{T_0} \io \gamma(\Theta) d_{xt} \vp_x
	- \int_0^{T_0} \io  [\gamma(\Theta)-\gamma(\wtt)] \wtu_{xt} \vp_x \nn\\
	& & - \int_0^{T_0} \io \hg(\Theta) d_x \vp_x
	- \int_0^{T_0} \io [\hg(\Theta)-\hg(\wtt)] \whu_x \vp_x.
  \eea
  To make appropriate use of this, we fix any $t_0\in (0,T_0)$, and for $\eta\in (0,T_0-t_0)$ we let
  $\zeta_\eta\in W^{1,\infty}(\R)$ be defined by
  \bas
	\zeta_\eta(t):=\lball
	1,
	\qquad & t\in (-\infty,t_0), \\[1mm]
	1-\frac{t-t_0}{\eta},
	\qquad & t\in [t_0,t_0+\eta], \\[1mm]
	0,
	\qquad & t>t_0+\eta.
	\ear
  \eas
  Then a straightforward approximation argument shows that (\ref{15.3}) actually extends so as to apply to
  \be{15.4}
	\vp(x,t) \equiv \vp_{\eta,h}(x,t)
	:=\zeta_\eta(t) \cdot \frac{1}{h} \int_{t-h}^t d_{tt}(x,s) ds,
	\qquad x\in\bom, \ t\in [0,T_0),
  \ee
  for any such $\eta$ and arbitrary $h\in (0,1)$, because $\vp_{\eta,h}\equiv 0$ in $\Om\times (t_0+\eta,T_0) \ne\emptyset$,
  and because (\ref{w2}) ensures that $d_{xt}$ and $d_x$ belong to $L^\infty(\Om\times (0,T_0))\subset L^1(\Om\times (0,T_0))$
  and $d_{xtt} \in L^\infty((0,T_0);L^2(\Om)) \subset L^1(\Om\times (0,T_0))$, so that
  $\vp_x\in L^1(\Om\times (0,T_0))$, and because furthermore
  \be{15.5}
	\vp_t(x,t)
	= \zeta_\eta(t)\cdot \frac{d_{tt}(x,t)-d_{tt}(x,t-h)}{h}
	+ \zeta_\eta'(t) \cdot \frac{1}{h} \int_{t-h}^t d_{tt}(x,s) ds
	\qquad \mbox{for all $x\in\Om$ and } t\in (0,T_0)
  \ee
  and hence, trivially, $\vp_t\in L^1(\Om\times (0,T_0))$.
  Since
  \bas
	\vp_x(x,t)
	= \zeta_\eta(t) \cdot \frac{d_{xt}(x,t)-d_{xt}(x,t-h)}{h}
	\qquad \mbox{for all $x\in\Om$ and } t\in (0,T_0),
  \eas
  from (\ref{15.3}) we thus obtain that
  \bea{15.6}
	& & \hs{-14mm}
	- \frac{1}{h} \int_0^{T_0} \io \zeta_\eta(t) d_{tt}^2(x,t) dxdt
	+ \frac{1}{h} \int_0^{T_0} \io \zeta_\eta(t) d_{tt}(x,t) d_{tt}(x,t-h) dxdt \nn\\
	& & - \int_0^{T_0} \io \zeta_\eta'(t) d_{tt}(x,t) \cdot \bigg\{ \frac{1}{h} \int_{t-h}^t d_{tt}(x,s) ds\bigg\} dxdt \nn\\
	& & + \al \int_0^{T_0} \io \zeta_\eta(t) d_{tt}(x,t) \cdot \bigg\{ \frac{1}{h} \int_{t-h}^t d_{tt}(x,s) ds\bigg\} dxdt \nn\\
	&=& - \frac{1}{h} \int_0^{T_0} \io \zeta_\eta(t) \gamma(\Theta(x,t)) d_{xt}^2(x,t) dxdt
	+ \frac{1}{h} \int_0^{T_0} \io \zeta_\eta(t) \gamma(\Theta(x,t)) d_{xt}(x,t) d_{xt}(x,t-h) dxdt \nn\\
	& & - \int_0^{T_0} \io \zeta_\eta(t) \cdot [\gamma(\Theta(x,t))-\gamma(\wtt(x,t))] \wtu_{xt}(x,t) dxdt \cdot
		\bigg\{ \frac{1}{h} \int_{t-h}^t d_{xtt}(x,s) ds\bigg\} dxdt \nn\\
	& & - \int_0^{T_0} \io \zeta_\eta(t) \hg(\Theta(x,t)) d_x(x,t) \cdot
		\bigg\{ \frac{1}{h} \int_{t-h}^t d_{xtt}(x,s) ds\bigg\} dxdt \nn\\
	& & - \int_0^{T_0} \io \zeta_\eta(t) \cdot [\hg(\Theta(x,t))-\hg(\wtt(x,t))] \wtu_x(x,t) \cdot
		\bigg\{ \frac{1}{h} \int_{t-h}^t d_{xtt}(x,s) ds \bigg\} dxdt
  \eea
  for all $\eta\in (0,T_0-t_0)$ and $h\in (0,1)$.
  Here, the continuity of $d_{tt}$ ensures that
  \bas
	\sup_{(x,t)\in\Om\times (0,T_0)} \bigg| \frac{1}{h} \int_{t-h}^t d_{tt}(x,s) ds - d_{tt}(x,t)\bigg| \to 0
	\qquad \mbox{as } h\searrow 0,
  \eas
  while in view of the inclusion $d_{xtt}\in L^2(\Om\times (-1,T_0))$, as implied by (\ref{w2}), a well-known approximation feature
  of Steklov averages (\cite{dibenedetto}) asserts that
  \bas
	\int_0^{t_0+\eta} \io \bigg| \frac{1}{h} \int_{t-h}^t d_{xtt}(x,s)ds - d_{xtt}(x,t) \bigg|^2 dxdt \to 0
	\qquad \mbox{as } h\searrow 0.
  \eas
  In (\ref{15.6}), we thus have
  \bea{15.7}
	& & \hs{-20mm}
	- \int_0^{T_0} \io \zeta_\eta'(t) d_{tt}(x,t) \cdot \bigg\{ \frac{1}{h} \int_{t-h}^t d_{tt}(x,s) ds \bigg\} dxdt
	+ \al \int_0^{T_0} \io \zeta_\eta(t) d_{tt}(x,t) \cdot \bigg\{ \frac{1}{h} \int_{t-h}^t d_{tt}(x,s) ds\bigg\} dxdt \nn\\
	&\to&
	- \int_0^{T_0} \io \zeta_\eta' d_{tt}^2
	+ \al \int_0^{T_0} \io \zeta_\eta d_{tt}^2
	\qquad \mbox{as } h\searrow 0
  \eea
  and
  \bea{15.8}
	& & \hs{-20mm}
	- \int_0^{T_0} \io \zeta_\eta(t) \cdot [\gamma(\Theta(x,t))-\gamma(\wtt(x,t))] \wtu_{xt}(x,t) dxdt \cdot
		\bigg\{ \frac{1}{h} \int_{t-h}^t d_{xtt}(x,s) ds\bigg\} dxdt \nn\\
	& & - \int_0^{T_0} \io \zeta_\eta(t) \hg(\Theta(x,t)) d_x(x,t) \cdot
		\bigg\{ \frac{1}{h} \int_{t-h}^t d_{xtt}(x,s) ds\bigg\} dxdt  \nn\\
	& & - \int_0^{T_0} \io \zeta_\eta(t) \cdot [\hg(\Theta(x,t))-\hg(\wtt(x,t))] \wtu_x(x,t) \cdot
		\bigg\{ \frac{1}{h} \int_{t-h}^t d_{xtt}(x,s) ds \bigg\} dxdt \nn\\
	&\to&
	- \int_0^{T_0} \io \zeta_\eta \cdot [\gamma(\Theta)-\gamma(\wtt)] \wtu_{xt} d_{xtt}
	- \int_0^{T_0} \io \zeta_\eta \hg(\Theta) d_x d_{xtt} \nn\\
	& & - \int_0^{T_0} \io \zeta_\eta \cdot [\hg(\Theta)-\hg(\wtt)] \wtu_x d_{xtt}
	\qquad \mbox{as } h\searrow 0.
  \eea
  In treating the first two summands in (\ref{15.6}), for $x\in\Om$, $t\in (0,T_0)$ and $h\in (0,1)$ we rewrite
  \bas
	- \frac{1}{h} d_{tt}^2(x,t) + \frac{1}{h} d_{tt}(x,t) d_{tt}(x,t-h)
	= -\frac{1}{2h} d_{tt}^2(x,t) + \frac{1}{2h} d_{tt}^2(x,t-h)
	- \frac{1}{2h} \cdot \big\{ d_{tt}(x,t)-d_{tt}(x,t-h)\big\}^2
  \eas
  to see by means of a substitution that
  \bea{15.9}
	& & \hs{-20mm}
	- \frac{1}{h} \int_0^{T_0} \io \zeta_\eta(t) d_{tt}^2(x,t) dxdt
	+ \frac{1}{h} \int_0^{T_0} \io \zeta_\eta(t) d_{tt}(x,t) d_{tt}(x,t-h) dxdt \nn\\
	&=& - \frac{1}{2h} \int_0^{T_0} \io \zeta_\eta(t) d_{tt}^2(x,t) dxdt
	+ \frac{1}{2h} \int_0^{T_0} \io \zeta_\eta(t) d_{tt}^2(x,t-h) dxdt \nn\\
	& & - \frac{1}{2h} \int_0^{T_0} \io \zeta_\eta(t) \cdot \big\{ d_{tt}(x,t)-d_{tt}(x,t-h)\big\}^2 dxdt \nn\\
	&=& - \frac{1}{2h} \int_{T_0-h}^{T_0} \io \zeta_\eta(t) d_{tt}^2(x,t) dxdt
	+ \frac{1}{2} \int_0^{T_0-h} \io \frac{\zeta_\eta(t+h)-\zeta_\eta(t)}{h} d_{tt}^2(x,t) dxdt \nn\\
	& & - \frac{1}{2h} \int_0^{T_0} \io \zeta_\eta(t) \cdot \big\{ d_{tt}(x,t)-d_{tt}(x,t-h)\big\}^2 dxdt
  \eea
  for all $\eta\in (0,T_0-t_0)$ and $h\in (0,1)$.
  Here, again by continuity of $d_{tt}$,
  \be{15.10}
	- \frac{1}{2h} \int_{T_0-h}^{T_0} \io \zeta_\eta(t) d_{tt}^2(x,t) dxdt
	\to - \frac{1}{2} \io \zeta_\eta(T_0) d_{tt}^2(x,T_0) dx = 0
	\qquad \mbox{as } h\searrow 0
  \ee
  and
  \bea{15.11}
	\frac{1}{2} \int_0^{T_0-h} \io \frac{\zeta_\eta(t+h)-\zeta_\eta(t)}{h} d_{tt}^2(x,t) dxdt
	\to \frac{1}{2} \int_0^{T_0} \io \zeta_\eta' d_{tt}^2
	\qquad \mbox{as } h\searrow 0
  \eea
  due to the fact that $\frac{\zeta_\eta(\cdot,+h)-\zeta_\eta}{h} \wsto \zeta_\eta'$ in $L^\infty((0,T_0))$ as $h\searrow 0$.
  Furthermore, using that $d_{ttt}\in L^\infty_{loc}([0,T);L^2(\Om))$
  according to the remark following Definition \ref{dw},
  by means of the Cauchy-Schwarz inequality we obtain that
  \bea{15.12}
	& & \hs{-30mm}
	\bigg| - \frac{1}{2h} \int_0^{T_0} \io \zeta_\eta(t) \cdot \big\{ d_{tt}(x,t)-d_{tt}(x,t-h)\big\}^2 dxdt  \bigg| \nn\\
	&=& \frac{1}{2h} \int_0^{T_0} \io \zeta_\eta(t) \cdot \bigg\{ \int_{t-h}^t d_{ttt}(x,s) ds\bigg\}^2 dxdt \nn\\
	&\le& \frac{1}{2} \int_0^{T_0} \io \zeta_\eta(t) \cdot \bigg\{ d_{ttt}^2(x,s) ds \bigg\} dxdt \nn\\
	&\le& \frac{T_0 h}{2} \cdot {\rm ess} \!\!\! \sup_{\hs{-3mm} t\in (0,t_0+\eta)} \io d_{ttt}^2(d,t) dx \nn\\[2mm]
	&\to& 0
	\qquad \mbox{as } h\searrow 0.
  \eea
  Finally, Young's inequality guarantees that on the right-hand side of (\ref{15.6}) we have
  \bea{15.13}
	& & \hs{-20mm}
	- \frac{1}{h} \int_0^{T_0} \io \zeta_\eta(t) \gamma(\Theta(x,t)) d_{xt}^2(x,t) dxdt
	+ \frac{1}{h} \int_0^{T_0} \io \zeta_\eta(t) \gamma(\Theta(x,t)) d_{xt}(x,t) d_{xt}(x,t-h) dxdt \nn\\
	&\le& - \frac{1}{2h} \int_0^{T_0} \io \zeta_\eta(t) \gamma(\Theta(x,t)) d_{xt}^2(x,t) dxdt
	+ \frac{1}{2h} \int_0^{T_0} \io \zeta_\eta(t) \gamma(\Theta(x,t)) d_{xt}^2(x,t-h) dxdt \nn\\
	&=& - \frac{1}{2h} \int_{T_0-h}^{T_0} \io \zeta_\eta(t) \gamma(\Theta(x,t)) d_{xt}^2(x,t) dxdt \nn\\
	& & + \frac{1}{2} \int_0^{T_0-h} \io \frac{\zeta_\eta(t+h)\gamma(\Theta(x,t+h))-\zeta_\eta(t)\gamma(\Theta(x,t))}{h} \cdot
		d_{xt}^2(x,t) dxdt \nn\\
	&\to& - \frac{1}{2} \io \zeta_\eta(T_0) \gamma(\Theta(x,T_0)) d_{xt}^2(x,T_0) dxdt
	+ \frac{1}{2} \int_0^{T_0} \io \pa_t \big\{ \zeta_\eta \gamma(\Theta)\big\} \cdot d_{xt}^2 \nn\\
	&=& \frac{1}{2} \int_0^{T_0} \io \zeta_\eta' \gamma(\Theta) d_{xt}^2
	+ \frac{1}{2} \int_0^{T_0} \io \zeta_\eta \gamma'(\Theta) \Theta_t d_{xt}^2
	\qquad \mbox{as } h\searrow 0.
  \eea
  Collecting (\ref{15.9})-(\ref{15.12}), (\ref{15.7}), (\ref{15.13}) and (\ref{15.8}), in the limit $h\searrow 0$ we thus infer from
  (\ref{15.6}) and our assumption that $\al\ge 0$ that for each $\eta\in (0,T_0-t_0)$,
  \bea{15.14}
	& & \hs{-30mm}
	- \frac{1}{2} \int_0^{T_0} \io \zeta_\eta' d_{tt}^2
	- \frac{1}{2} \int_0^{T_0} \io \zeta_\eta' \gamma(\Theta) d_{xt}^2 \nn\\
	&\le& \frac{1}{2} \int_0^{T_0} \io \zeta_\eta \gamma'(\Theta) \Theta_t d_{xt}^2
	- \int_0^{T_0} \io \zeta_\eta \cdot [\gamma(\Theta)-\gamma(\wtt)] \wtu_{xt} d_{xtt} \nn\\
	& & - \int_0^{T_0} \io \zeta_\eta \hg(\Theta) d_x d_{xtt}
	- \int_0^{T_0} \io \zeta_\eta \cdot [\hg(\Theta)-\hg(\wtt)] \wtu_x d_{xtt},
  \eea
  where we once more use that according to Definition \ref{dw}, $d_{tt}$ lies in $C^0(\bom\times [0,T_0])$ with
  $d_{xtt} \in L^2(\Om\times (0,T_0))$.
  Therefore, namely, two integrations by parts with respect to the space variable show that
  \bea{15.15}
	\int_0^{T_0} \io \zeta_\eta \cdot [\gamma(\Theta)-\gamma(\wtt)] \wtu_{xt} d_{xtt}
	&=& \int_0^{T_0} \io \zeta_\eta \cdot [\gamma'(\Theta) \Theta_x - \gamma'(\wtt)\wtt_x] \wtu_{xt} d_{xt} \nn\\
	& & + \int_0^{T_0} \io \zeta_\eta \cdot [\gamma(\Theta)-\gamma(\wtt)] \wtu_{xxt} d_{xt}
  \eea
  and, similarly,
  \bea{15.16}
	- \int_0^{T_0} \io \zeta_\eta \cdot [\hg(\Theta)-\hg(\wtt)] \wtu_x d_{xtt}
	&=& \int_0^{T_0} \io \zeta_\eta \cdot [\hg'(\Theta)\Theta_x - \hg'(\wtt)\wtt_x] \wtu_x d_{xt} \nn\\
	& & + \int_0^{T_0} \io \zeta_\eta \cdot [\hg(\Theta)-\hg(\wtt)] \wtu_{xx} d_{xt},
  \eea
  while furthermore using that the continuous function $\zeta_\eta \hg(\Theta) d_x d_{xt}$ vanishes on $\bom$ at $t=0$
  and for all $t\in (t_0+\eta,T_0)$ we may integrate by parts in time to see that
  \bea{15.17}
	- \int_0^{T_0} \io \zeta_\eta \hg(\Theta) d_x d_{xt}
	&=& \int_0^{T_0} \io \zeta_\eta \hg(\Theta) d_{xt}^2
	+ \int_0^{T_0} \io \zeta_\eta \hg'(\Theta) \Theta_t d_x d_{xt} \nn\\
	& & + \int_0^{T_0} \io \zeta_\eta' \hg(\Theta) d_x d_{xt}.
  \eea
  To suitably estimate the right-hand sides in (\ref{15.15})-(\ref{15.17}), we rely on the fact that $T_0<T$ in choosing
  $c_i=c_i(T_0)>0$, $i\in\{1,2,...,15\}$, such that in line with Definition \ref{dw} we have
  \be{15.18}
	\hs{-5mm}
	\|\wtu_{xt}\|_{L^\infty(\Om)} \le c_1,
	\
	\|\wtu_{xxt}\|_{L^\infty(\Om)} \le c_2,
	\
	\|\wtu_x\|_{L^\infty(\Om)} \le c_3
	\ \mbox{and} \
	\|\wtu_{xx}\|_{L^2(\Om)} \le c_4
	\quad \mbox{for a.e.~} t\in (0,T_0)
  \ee
  as well as
  \be{15.19}
	\|\Theta\|_{L^\infty(\Om)}
	+ \|\wtt\|_{L^\infty(\Om)} \le c_5,
	\quad
	\|\Theta_x\|_{L^\infty(\Om)} \le c_6,
	\quad
	\|\wtt_x\|_{L^\infty(\Om)} \le c_7
	\quad \mbox{and} \quad
	\|\Theta_t\|_{L^\infty(\Om)} \le c_8
  \ee
  for all $t\in (0,T_0)$,
  and that
  \bea{15.20}
	c_9 \le \gamma(\xi)\le c_{10}
	\quad \mbox{and} \quad
	\hg(\xi) \le c_{11}
	\qquad \mbox{for all } \xi\in [0,c_5]
  \eea
  as well as
  \bea{15.21}
	|\gamma'(\xi)| \le c_{12},
	\quad
	|\gamma''(\xi)| \le c_{13},
	\quad
	|\hg'(\xi)| \le c_{14}
	\quad \mbox{and} \quad
	|\hg''(\xi)| \le c_{15}
	\qquad \mbox{for all } \xi\in [0,c_5].
  \eea
  We furthermore note that a combination of the mean value theorem with (\ref{15.19}) and (\ref{15.21}) shows that
  \bas
	\big| \gamma(\Theta) - \gamma(\wtt)\big| \le c_{12} |\Theta-\wtt| = c_{12} |\del|
	\qquad \mbox{in } \Om\times (0,T_0)
  \eas
  and
  \bas
	\big|\hg(\Theta)-\hg(\wtt)\big| \le c_{14} |\del|
	\qquad \mbox{in } \Om\times (0,T_0),
  \eas
  and that
  \bas
	\big|\gamma'(\Theta)\Theta_x - \gamma'(\wtt)\wtt_x\big|
	&\le& \big| \gamma'(\Theta)-\gamma'(\wtt)\big| \cdot |\Theta_x|
	+ \big|\gamma'(\wtt)\big| \cdot |\Theta_x-\wtt_x| \nn\\
	&\le& c_{13} c_6 |\del|
	+ c_{12} |\del_x|
	\qquad \mbox{in } \Om\times (0,T_0)
  \eas
  as well as
  \bas
	\big| \hg'(\Theta) \Theta_x - \hg'(\wtt)\wtt_x\big|
	&\le& \big| \hg'(\Theta) - \hg'(\wtt)\big| \cdot |\Theta_x|
	+ \big| \hg'(\wtt)\big| \cdot |\Theta_x - \wtt_x| \nn\\
	&\le& c_{15} c_6 |\del| + c_{14} |\del_x|
	\qquad \mbox{in } \Om\times (0,T_0).
  \eas
  From (\ref{15.15}), (\ref{15.16}) (\ref{15.18}) and the fact that $\supp \zeta_\eta \subset [0,t_0+\eta] \subset [0,T_0]$
  for all $\eta\in (0,T_0-t_0)$, we accordingly obtain that due to H\"older's and Young's inequalities,
  \bea{15.22}
	\bigg| - \int_0^{T_0} \io \zeta_\eta \cdot [\gamma(\Theta)-\gamma(\wtt)] \wtu_{xt} d_{xtt} \bigg|
	&\le& \int_0^{t_0+\eta} \io |\gamma'(\Theta)\Theta_x - \gamma'(\wtt)\wtt_x| \cdot |\wtu_{xt}| \cdot |d_{xt}| \nn\\
	& & + \int_0^{t_0+\eta} \io |\gamma(\Theta)-\gamma(\wtt)| \cdot |\wtu_{xxt}| \cdot |d_{xt}| \nn\\
	&\le& c_{13} c_6 c_1 \int_0^{t_0+\eta} \io |\del| \cdot |d_{xt}|
	+ c_{12} c_1 \int_0^{t_0+\eta} \io |\del_x| \cdot |d_{xt}| \nn\\
	& & + c_{12} \int_0^{t_0+\eta} \|\del(\cdot,t)\|_{L^\infty(\Om)} \|\wtu_{xxt}(\cdot,t)\|_{L^2(\Om)}
		\|d_{xt}(\cdot,t)\|_{L^2(\Om)} dt \nn\\
	&\le& 2 \int_0^{t_0+\eta} \io d_{xt}^2
	+ \frac{c_{13}^2 c_6^2 c_1^2}{4} \int_0^{t_0+\eta} \io \del^2
	+ \frac{c_{12}^2 c_1^2}{4} \int_0^{t_0+\eta} \io d_x^2 \nn\\
	& & + c_{12} c_2 \int_0^{t_0+\eta} \|\del(\cdot,t)\|_{L^\infty(\Om)} \|d_{xt}(\cdot,t)\|_{L^2(\Om)} dt \nn\\
	&\le& 3 \int_0^{t_0+\eta} \io d_{xt}^2
	+ \frac{c_{13}^2 c_6^2 c_1^2}{4} \int_0^{t_0+\eta} \io \del^2
	+ \frac{c_{12}^2 c_1^2}{4} \int_0^{t_0+\eta} \io d_x^2 \nn\\
	& & + \frac{c_{12}^2 c_2^2}{4} \int_0^{t_0+\eta} \|\del(\cdot,t)\|_{L^\infty(\Om)}^2 dt,
  \eea
  and that, similarly,
  \bea{15.23}
	\bigg| - \int_0^{T_0} \io \zeta_\eta \cdot [\hg(\Theta)-\hg(\wtt)] \wtu_x d_{xtt} \bigg|
	&\le&c_{15} c_6 c_3 \int_0^{t_0+\eta} \io |\del|\cdot |d_{xt}|
	+ c_{14} c_3 \int_0^{t_0+\eta} \io |\del_x| \cdot |d_{xt}| \nn\\
	& & + c_{14} \int_0^{t_0+\eta} \|\del(\cdot,t)\|_{L^\infty(\Om)} \|\wtu_{xx}(\cdot,t)\|_{L^2(\Om)}
		\|d_{xt}(\cdot,t)\|_{L^2(\Om)} dt \nn\\
	&\le& 3\int_0^{t_0+\eta} \io d_{xt}^2
	+ \frac{c_{15}^2 c_6^2 c_3^2}{4} \int_0^{t_0+\eta} \io \del^2 \nn\\
	& & + \frac{c_{14}^2 c_3^2}{4} \int_0^{t_0+\eta} \io \del_x^2
	+ \frac{c_{14}^2 c_4^2}{4} \int_0^{t_0+\eta} \|\del(\cdot,t)\|_{L^\infty(\Om)}^2 dt,
  \eea
  while Young's inequality together with (\ref{15.20}) and (\ref{15.19}) implies that
  \bea{15.24}
	& & \hs{-30mm}
	- \int_0^{T_0} \io \zeta_\eta \hg(\Theta) d_x d_{xtt}
	- \int_0^{T_0} \io \zeta_\eta' \hg(\Theta) d_x d_{xt} \nn\\
	&\le& c_{11} \int_0^{t_0+\eta} \io d_{xt}^2
	+ c_{14} c_8 \int_0^{t_0+\eta} \io |d_x| \cdot |d_{xt}| \nn\\
	&\le& (c_{11}+1) \int_0^{t_0+\eta} \io d_{xt}^2
	+ \frac{c_{14}^2 c_8^2}{4} \int_0^{t_0+\eta} \io d_x^2.
  \eea
  Taking $c_{16}>0$ such that in line with the continuity of the embedding $W^{1,2}(\Om)\hra L^\infty(\Om)$ we have
  \bas
	\|\vp\|_{L^\infty(\Om)}^2 \le c_{16} \io \vp_x^2 + c_{16} \io \vp^2
	\qquad \mbox{for all } \vp\in W^{1,2}(\Om),
  \eas
  from (\ref{15.22})-(\ref{15.24}) we thus infer that (\ref{15.14}) implies the inequality
  \bea{15.25}
	& & \hs{-20mm}
	- \frac{1}{2} \int_0^{T_0} \io \zeta_\eta' d_{tt}^2
	- \frac{1}{2} \int_0^{T_0} \io \zeta_\eta' \gamma(\Theta) d_{xt}^2
	- \int_0^{T_0} \io \zeta_\eta' \hg(\Theta) d_x d_{xt} \nn\\
	&\le& c_{17} \int_0^{t_0+\eta} \io d_{xt}^2
	+ c_{17} \int_0^{t_0+\eta} \io  d_x^2
	+ c_{17} \int_0^{t_0+\eta} \io \del_x^2
	+ c_{17} \int_0^{t_0+\eta} \io \del^2
  \eea
  with
  \bas
	c_{17}:=\max \Big\{ c_{11}+7 \, , \,
		\frac{c_{13}^2 c_6^2 c_1^2}{4}+\frac{c_{12}^2 c_2^2 c_{16}}{4}+\frac{c_{15}^2 c_6^2 c_3^2}{4}
		+ \frac{c_{14}^2 c_4^2 c_{16}}{4} \, , \,
	\frac{c_{12}^2 c_1^2}{4} + \frac{c_{12}^2 c_2^2 c_{16}}{4} + \frac{c_{14}^2 c_3^2}{4}
		+ \frac{c_{14}^2 c_4^2 c_{16}}{4}
	\Big\}.
  \eas
  Now by continuity of $d_{tt},\Theta,d_{xt}$ and $d_x$, on the left-hand side of (\ref{15.25}) it follows from our definition
  of $(\zeta_\eta)_{\eta\in (0,T_0-t_0)}$ that
  \bas
	- \frac{1}{2} \int_0^{T_0} \io \zeta_\eta' d_{tt}^2
	= \frac{1}{2\eta} \int_{t_0}^{t_0+\eta} \io d_{tt}^2
	\to \frac{1}{2} \io d_{tt}^2(\cdot,t_0)
	\qquad \mbox{as } \eta\searrow 0,
  \eas
  and that similarly, as $\eta\searrow 0$ we have
  \bas
	- \frac{1}{2} \int_0^{T_0} \io \zeta_\eta' \gamma(\Theta) d_{xt}^2
	\to \frac{1}{2} \io \gamma\big(\Theta(\cdot,t_0)\big) d_{xt}^2(\cdot,t_0)
  \eas
  and
  \bas
	- \int_0^{T_0} \io \zeta_\eta' \hg(\Theta) d_x d_{xt}
	\to \io \hg\big(\Theta(\cdot,t_0)\big) d_x(\cdot,t_0) d_{xt}(\cdot,t_0),
  \eas
  whence from (\ref{15.25}) we obtain on letting $\eta\searrow 0$ that
  \bea{15.26}
	& & \hs{-30mm}
	\frac{1}{2} \io d_{tt}^2(\cdot,t_0)
	+ \frac{1}{2} \io \gamma\big(\Theta(\cdot,t_0)\big) d_{xt}^2(\cdot,t_0)
	+ \io \hg\big(\Theta(\cdot,t_0)\big) d_x(\cdot,t_0) d_{xt}(\cdot,t_0) \nn\\
	&\le& c_{17} \int_0^{t_0} \io d_{xt}^2
	+ c_{17} \int_0^{t_0} \io  d_x^2
	+ c_{17} \int_0^{t_0} \io \del_x^2
	+ c_{17} \int_0^{t_0} \io \del^2
  \eea
  for all $t_0\in (0,T_0)$.
  Here, the integrals on the right-hand side can adequately be compensated on the basis of the simple observation that in line with
  the continuity of $d_x$ and $d_{xt}$, $(0,T_0) \ni t \mapsto \io d_x^2$ is continuously differentiable with
  \be{15.27}
	\frac{1}{2} \frac{d}{dt} \io d_x^2
	= \io d_x d_{xt}
	\le \io d_{xt}^2 + \frac{1}{4} \io d_x^2
	\qquad \mbox{for all } t\in (0,T_0),
  \ee
  and that since
  \bas
	\del_t
	&=& D\del_{xx} + \Gamma(\Theta) u_{xt}^2 - \Gamma(\wtt) \wtu_{xt}^2 \\
	&=& D\del_{xx} + [\Gamma(\Theta)-\Gamma(\wtt)] \wtu_{xt}^2
	+ \Gamma(\Theta) (u_{xt} + \wtu_{xt}) d_{xt}
  \eas
  holds in the classical sense in $\Om\times (0,T_0)$, we have
  \be{15.28}
	\frac{1}{2} \frac{d}{dt} \io \del^2
	+ D \io \del_x^2
	= \io [\Gamma(\Theta)-\Gamma(\wtt)] \wtu_{xt}^2 \cdot \del
	+ \io \Gamma(\Theta) (u_{xt}+\wtu_{xt}) \cdot d_{xt} \del
	\qquad \mbox{for all } t\in (0,T_0).
  \ee
  If, beyond the choices in (\ref{15.18})-(\ref{15.21}), we let $c_{18}>0, c_{19}>0$ and $c_{20}>0$ be such that
  \bas
	\|u_{xt}(\cdot,t)\|_{L^\infty(\Om)} \le c_{18}
	\qquad \mbox{as well as} \qquad
	|\Gamma(\xi)| \le c_{19}
	\quad \mbox{and} \quad
	|\Gamma'(\xi)| \le c_{20}
	\quad \mbox{for all } \xi\in [0,c_5],
  \eas
  then combining (\ref{15.28}) with (\ref{15.18}) and (\ref{15.19}) we see that due to Young's inequality,
  \bas
	\frac{1}{2} \frac{d}{dt} \io \del^2
	+ D \io \del_x^2
	&\le& c_{20} c_1^2 \io \del^2
	+ c_{19} (c_{18}+c_1) \io |d_{xt}| \cdot |\del| \nn\\
	&\le& \io d_{xt}^2
	+ c_{21} \io \del^2
	\qquad \mbox{for all } t\in (0,T_0)
  \eas
  with $c_{21}:=c_{20} c_1^2 + \frac{c_{19}^2 (c_{18}+c_1)^2}{4}$.
  On integrating this and (\ref{15.27}) in time, we thus infer that
  \bas
	y(t):=\frac{1}{2} \io d_{tt}^2
	+ \frac{1}{2} \io \gamma(\Theta) d_{xt}^2
	+ \io \hg(\Theta) d_x d_{xt}
	+ \frac{B_1}{2} \io d_x^2
	+ \frac{B_2}{2} \io \del^2,
	\qquad t\in [0,T_0],
  \eas
  with
  \be{15.288}
	B_1:=\frac{4 c_{11}^2}{c_9}
	\qquad \mbox{and} \qquad
	B_2:=\frac{c_{17}}{D},
  \ee
  satisfies
  \bas
	y(t_0)
	+ B_2 D \int_0^{t_0} \io \del_x^2
	&\le& c_{17} \int_0^{t_0} \io d_{xt}^2
	+ c_{17} \int_0^{t_0} \io d_x^2
	+ c_{17} \int_0^{t_0} \io \del_x^2
	+ c_{17} \int_0^{t_0} \io \del^2 \nn\\
	& & + B_1 \int_0^{t_0} \io d_{xt}^2
	+ \frac{B_1}{4} \int_0^{t_0} \io d_x^2
	+ B_2 \int_0^{t_0} \io d_{xt}^2
	+ B_2 c_{21} \int_0^{t_0} \io \del^2
  \eas
  for all $t\in (0,T_0)$ and hence, due to the second definition in (\ref{15.288}),
  \be{15.29}
	y(t_0)
	\le c_{22} \int_0^{t_0} \io d_{xt}^2
	+ c_{22} \int_0^{t_0} \io d_x^2
	+ c_{22} \int_0^{t_0} \io \del^2
	\qquad \mbox{for all } t\in (0,T_0)
  \ee
  with $c_{22}:=\max\{ c_{17} + B_1+B_2 \, , \, c_{17} + B_2 c_{21}\}$.
  As, on the other hand, the first selection made in (\ref{15.288}) together with Young's inequality and (\ref{15.20}) ensures that
  \bea{15.30}
	y(t)
	&\ge& \frac{1}{2} \io d_{tt}^2
	+ \frac{c_9}{2} \io d_{xt}^2
	- c_{11} \io |d_x| \cdot |d_{xt}|
	+ \frac{B_1}{2} \io d_x^2
	+ \frac{B_2}{2} \io \del^2 \nn\\
	&\ge& \frac{1}{2} \io d_{tt}^2
	+ \frac{c_9}{4} \io d_{xt}^2
	- \frac{c_{11}^2}{c_9} \io d_x^2
	+ \frac{B_1}{2} \io d_x^2
	+ \frac{B_2}{2} \io \del^2 \nn\\
	&=& \frac{1}{2} \io d_{tt}^2
	+ \frac{c_9}{4} \io d_{xt}^2
	+ \frac{B_1}{4} \io d_x^2
	+ \frac{B_2}{2} \io \del^2
	\qquad \mbox{for all } t\in (0,T_0),
  \eea
  from (\ref{15.29}) we particularly obtain that
  \bas
	y(t_0) \le c_{23} \int_0^{t_0} y(t) dt
	\qquad \mbox{for all } t\in (0,T_0),
  \eas
  where $c_{23}:=c_{22} \cdot \max\{ \frac{4}{c_9} \, , \, \frac{4}{B_1} \, , \, \frac{2}{B_2}\}$.
  By continuity and nonnegativity of $y$ on $[0,T_0]$, the latter especially being entailed by (\ref{15.30}),
  through a Gr\"onwall lemma this implies that $y(t_0)=0$ for all $t_0\in (0,T_0)$, whence again relying on (\ref{15.30})
  we may conclude that $d_{tt}\equiv 0$ and $\del\equiv 0$ in $\Om\times (0,T_0)$.
  Since $d(x,0)=d_t(x,0)=0$ for all $x\in\Om$ by Definition \ref{dw}, and since $T_0\in (0,T)$ was arbitrary, this
  means that, indeed, $u-\wtu\equiv 0$ and $\Theta-\wtt\equiv 0$ in $\Om\times (0,T)$.
\qed
\mysection{Existence and extensibility for initial data with $\io u_0 = \io u_{0t}= \io u_{0tt}=0$}
Next concerned with our claim concerning the existence of solutions,
in order to address this in a conveniently simple setting essentially equivalent to that covered by Theorem \ref{theo16}
we shall assume throughout this section that the open bounded interval $\Om$ is fixed, that $D>0$ and $\al\ge 0$,
that $\gamma, \hg$ and $\Gamma$ comply with (\ref{g}), and that the initial data $u_0, u_{0t}, u_{0tt}$ and $\Theta_0$
are such that beyond (\ref{init}) we additionally have
\be{mass0}
	\io u_0 = \io u_{0t} = \io u_{0tt} =0.
\ee
\subsection{Approximate problems and their local solvability}
In order to suitably approximate (\ref{0}), by using standard smoothing procedures we can fix families
$(\gaeps)_{\eps\in (0,1)} \subset C^\infty([0,\infty))$,
$(\hgeps)_{\eps\in (0,1)} \subset C^\infty([0,\infty))$ and
$(\Gaeps)_{\eps\in (0,1)} \subset C^\infty([0,\infty))$
as well as
$(u_{0\eps})_{\eps\in (0,1)} \subset C^\infty(\bom)$,
$(v_{0\eps})_{\eps\in (0,1)} \subset C^\infty(\bom)$,
$(w_{0\eps})_{\eps\in (0,1)} \subset C^\infty(\bom)$,
$(\Theta_{0\eps})_{\eps\in (0,1)} \subset C^\infty(\bom)$
in such a way that
\be{gaeps}
	\gaeps>0, \quad \hgeps>0 \quad \mbox{and} \quad \Gaeps\ge 0
	\quad \mbox{on $[0,\infty)$ \qquad for all } \eps\in (0,1)
\ee
and
\be{gaepsc}
	\gaeps \to \gamma
	\mbox{ in } C^2_{loc}([0,\infty)),
	\quad
	\hgeps \to \hg
	\mbox{ in } C^2_{loc}([0,\infty))
	\quad \mbox{and} \quad
	\Gaeps \to \Gamma
	\mbox{ in } C^1_{loc}([0,\infty))
	\qquad \mbox{as } \eps\searrow 0,
\ee
that $u_{0\eps x}$, $v_{0\eps x}$, $w_{0\eps x}$ and $\Theta_{0\eps x}$ have compact support in $\Om$ with
\be{mass0e}
	\io u_{0\eps} = \io v_{0\eps} = \io w_{0\eps} =0
\ee
and $\Theta_{0\eps} \ge 0$ in $\Om$
for all $\eps\in (0,1)$,
and that as $\eps\searrow 0$ we have
\be{iec}
	\lbal
	u_{0\eps} \to u_0
	\quad \mbox{ in } W^{2,2}(\Om),
	\\[1mm]
	v_{0\eps} \to u_{0t}
	\quad \mbox{ in } W^{2,2}(\Om),
	\\[1mm]
	w_{0\eps} \to u_{0tt}
	\quad \mbox{ in } W^{1,2}(\Om)
	\qquad \mbox{and}
	\\[1mm]
	\Theta_{0\eps} \wsto \Theta_0
	\quad \mbox{ in } W^{2,\infty}(\Om)
	\ear
\ee
as well as
\be{iec2}
	\sqrt{\eps} u_{0\eps xxx} \to 0
	\quad \mbox{in } L^2(\Om).
\ee
For each $\eps\in (0,1)$, the regularization of (\ref{0}) given by
\be{0eps}
	\lball
	\wepst = \eps \wepsxx + \big(\gaeps(\Teps) \vepsx\big)_x + \big(\hgeps(\Teps)\uepsx\big)_x - \al\weps,
	\qquad & x\in\Om, \ t>0, \\[1mm]
	\vepst = \eps \vepsxx + \weps,
	\qquad & x\in\Om, \ t>0, \\[1mm]
	\uepst = \eps \uepsxx + \veps,
	\qquad & x\in\Om, \ t>0, \\[1mm]
	\Tepst = D \Tepsxx + \Gaeps(\Teps) \vepsx^2,
	\qquad & x\in\Om, \ t>0, \\[1mm]
	\wepsx=\vepsx=\uepsx=\Tepsx=0,
	\qquad & x\in\pO, \ t>0, \\[1mm]
	\weps(x,0)=w_{0\eps}(x), \ \veps(x,0)=v_{0\eps}(x), \ \ueps(x,0)=u_{0\eps}(x), \ \Teps(x,0)=\Theta_{0\eps}(x),
	\quad & x\in\Om,
	\ear
\ee
then becomes accessible to standard arguments from local existence and extensibility in parabolic systems of cross-diffusion type,
accordingly leading to the following basic statement in this regard.
\begin{lem}\label{lem1}
  Let $\eps\in (0,1)$. Then there exist $\tme\in (0,\infty]$ and
  \be{reg}
	\lbal
	\weps \in C^{2,1}(\bom\times [0,\tme)) \cap C^\infty(\bom\times (0,\tme)), \\[1mm]
	\veps \in C^{2,1}(\bom\times [0,\tme)) \cap C^\infty(\bom\times (0,\tme)), \\[1mm]
	\ueps \in C^{2,1}(\bom\times [0,\tme)) \cap C^\infty(\bom\times (0,\tme))
	\qquad \mbox{and} \\[1mm]
	\Teps \in C^{2,1}(\bom\times [0,\tme)) \cap C^\infty(\bom\times (0,\tme))
	\ear
  \ee
  such that $\Teps\ge 0$ in $\Om\times (0,\tme)$, that $(\weps,\veps,\ueps,\Teps)$ solves
  (\ref{0eps}) classically in $\Om\times (0,\tme)$, and that
  \bea{ext_eps}
	& & \hs{-15mm}
	\mbox{if $\tme<\infty$, \quad then \quad} \nn\\
	& & \hs{-4mm}
	\limsup_{t\nearrow\tme} \Big\{
	\|\weps(\cdot,t)\|_{L^\infty(\Om)}
	+ \|\veps(\cdot,t)\|_{W^{1,2}(\Om)}
	+ \|\ueps(\cdot,t)\|_{W^{1,2}(\Om)}
	+ \|\Teps(\cdot,t)\|_{L^\infty(\Om)}
	\Big\}
	= \infty.
  \eea
  Moreover,
  \be{mass}
	\io \weps(\cdot,t) = \io \veps(\cdot,t) = \io \ueps(\cdot,t) =0
	\qquad \mbox{for all } t\in (0,\tme).
  \ee
\end{lem}
\proof
  We fix $\eps\in (0,1)$, and first pick any $\gamma_0 \in C^\infty(\R), \wt{\gamma}_0\in C^\infty(\R)$ and $\Gamma_0\in C^\infty(\R)$
  such that $\gamma_0|_{[0,\infty)} = \gaeps, \wt{\gamma}_0|_{[0,\infty)}=\hgeps$ and $\Gamma_0|_{[0,\infty)}=\Gaeps$.
  Next, according to known smoothing properties of the Neumann heat semigroup $(e^{t\Del})_{t\ge 0}$ on $\Om$ (\cite{win_JDE2010})
  we can find positive constants $c_1, c_2$ and $c_3$ such that for each $t\in (0,1)$,
  \be{l1}
	\|e^{\eps t\Del} \vp\|_{W^{1,2}(\Om)} \le c_1 t^{-\frac{1}{2}} \|\vp\|_{L^\infty(\Om)}
	\qquad \mbox{for all } \vp \in C^0(\bom)
  \ee
  and
  \be{l2}
	\|e^{d t\Del}\vp\|_{L^\infty(\Om)} \le c_2 t^{-\frac{1}{2}} \|\vp\|_{L^1(\Om)}
	\qquad \mbox{for all } \vp \in L^1(\Om),
  \ee
  and that
  \be{l3}
	\|e^{\eps t\Del} \pa_x \vp\|_{L^\infty(\Om)}
	\le c_3 t^{-\frac{3}{4}} \|\vp\|_{L^2(\Om)}
	\qquad \mbox{for all } \vp \in C_0^\infty(\Om),
  \ee
  meaning that $e^{\eps t\Del} \pa_x$ can be extended to a continuous linear operator on all of $L^2(\Om)$, with norm controlled
  according to (\ref{l3}).
  Here and throughout this proof, for definiteness we choose the norm in $W^{1,2}(\Om)$ to be given by
  $\|\vp\|_{W^{1,2}(\Om)} := \big\{ \|\vp\|_{L^2(\Om)}^2 + \|\vp_x\|_{L^2(\Om)}^2 \big\}^\frac{1}{2}$ for $\vp\in W^{1,2}(\Om)$,
  and we note that then two basic testing procedures and a comparison argument show that for any $t>0$,
  \be{l4}
	\|e^{t\Del}\vp\|_{W^{1,2}(\Om)} \le \|\vp\|_{W^{1,2}(\Om)}
	\qquad \mbox{for all } \vp \in W^{1,2}(\Om)
  \ee
  and
  \be{l5}
	\|e^{t\Del}\vp\|_{L^\infty(\Om)} \le \|\vp\|_{L^\infty(\Om)}
	\qquad \mbox{for all } \vp\in C^0(\bom).
  \ee
  Given $\eps\in (0,1)$, we now let
  $X_0:=C^0(\bom) \times W^{1,2}(\Om) \times W^{1,2}(\Om) \times C^0(\bom)$
  be equipped with the norm defined by
  \bas
	\big\| (\vp_1,\vp_2,\vp_3,\vp_4)\big\|_{X_0} := \max \Big\{
		\|\vp_1\|_{L^\infty(\Om)} \, , \, \|\vp_2\|_{W^{1,2}(\Om)} \, , \, \|\vp_3\|_{W^{1,2}(\Om)} \, , \,
		\|\vp_4\|_{L^\infty(\Om)} \Big\}
  \eas
  for $(\vp_1,\vp_2,\vp_3,\vp_4) \in X_0$, and abbreviating
  \be{l55}
	R:=\big\| (w_{0\eps},v_{0\eps},u_{0\eps},\Theta_{0\eps})\big\|_{X_0} +1
  \ee
  we fix $T_0=T_0(R)\in (0,1)$ small enough fulfilling
  \be{l6}
	4c_3 \cdot \big\{ \|\gamma_0\|_{L^\infty([-R,R])} + \|\wt{\gamma}_0\|_{L^\infty([-R,R])}\big\} \cdot R T_0^\frac{1}{4}
	+ \al R T \le 1
  \ee
  and
  \be{l7}
	2c_1 R T_0^\frac{1}{2} \le 1
  \ee
  and
  \be{l8}
	RT_0 \le 1
  \ee
  as well as
  \be{l9}
	2c_2 \|\Gamma_0\|_{L^\infty([-R,R])} R^2 T_0^\frac{1}{2} \le 1.
  \ee
  Then for $T\in (0,T_0]$, in the Banach space $X:=C^0([0,T];X_0)$, with norm
  given by $\|\vp\|_X := \sup_{t\in [0,T]} \|\vp(\cdot,t)\|_{X_0}$ for $\vp\in X$,
  we consider the closed set $S:=\{\vp\in X \ | \ \|\vp\|_X \le R\}$ and introduce a mapping
  $\Phi=(\Phi_1,\Phi_2,\Phi_3,\Phi_4): S\to X$ by defining
  \bas
	\big[ \Phi_1(w,v,u,\Theta)\big](\cdot,t)
	:= e^{\eps t\Del} w_{0\eps}
	+ \int_0^t e^{\eps(t-s)\Del} \pa_x \big\{ \gamma_0(\Theta) v_x + \wt{\gamma}_0(\Theta) u_x \big\} ds
	- \al \int_0^t e^{\eps(t-s)\Del} w ds
  \eas
  and
  \bas
	\big[ \Phi_2(w,v,u,\Theta)\big](\cdot,t)
	:= e^{\eps t\Del} v_{0\eps}
	+ \int_0^t e^{\eps(t-s)\Del} w ds
  \eas
  and
  \bas
	\big[ \Phi_3(w,v,u,\Theta)\big](\cdot,t)
	:= e^{\eps t\Del} v_{0\eps}
	+ \int_0^t e^{\eps(t-s)\Del} v ds
  \eas
  as well as
  \bas
	\big[ \Phi_4(w,v,u,\Theta)\big](\cdot,t)
	:= e^{D t\Del} \Theta_{0\eps}
	+ \int_0^t e^{D(t-s)\Del} \big\{ \Gamma_0(\Theta) v_x^2 \big\} ds
  \eas
  for $(w,v,u,\Theta) \in S$ and $t\in [0,T]$.
  Then by (\ref{l5}), (\ref{l3}) and (\ref{l6}),
  \bas
	& & \hs{-16mm}
	\big\| \big[ \Phi_1(w,v,u,\Theta)\big](\cdot,t) \big\|_{L^\infty(\Om)} \\
	&\le& \|w_{0\eps}\|_{L^\infty(\Om)}
	+ c_3 \int_0^t (t-s)^{-\frac{3}{4}} \cdot \big\{ \|\gamma_0(\Theta) v_x\|_{L^2(\Om)} + \|\wt{\gamma}_0(\Theta) u_x\|_{L^2(\Om)}
		\big\} ds
	+ \al \int_0^t \|w\|_{L^\infty(\Om)} ds \\
	&\le& \|w_{0\eps}\|_{L^\infty(\Om)}
	+ c_3 \cdot \big\{ \|\gamma_0\|_{L^\infty([-R,R])} + \|\wt{\gamma}_0\|_{L^\infty([-R,R])} \big\} \cdot R
		\int_0^t (t-s)^{-\frac{3}{4}} ds
	+ \al R T \\
	&\le& \|w_{0\eps}\|_{L^\infty(\Om)}
	+ 4 c_3 \cdot \big\{ \|\gamma_0\|_{L^\infty([-R,R])} + \|\wt{\gamma}_0\|_{L^\infty([-R,R])} \big\} \cdot R T^\frac{1}{4}
	+ \al R T \\
	&\le& R
	\qquad \mbox{for all } t\in [0,T],
  \eas
  because $T\le T_0 \le 1$, and because $\|w_{0\eps}\|_{L^\infty(\Om)} \le R-1$ according to (\ref{l55}).
  Similarly, (\ref{l1}) and (\ref{l4}) together with (\ref{l7}), (\ref{l8}) and (\ref{l55}) ensure that
  \bas
	\big\| \big[ \Phi_2(w,v,u,\Theta)\big](\cdot,t) \big\|_{W^{1,2}(\Om)}
	&\le& \|v_{0\eps}\|_{L^\infty(\Om)}
	+ c_1 \int_0^t (t-s)^{-\frac{1}{2}} \|w\|_{L^\infty(\Om)} ds \\
	&\le& \|v_{0\eps}\|_{L^\infty(\Om)}
	+ c_1 R \int_0^t (t-s)^{-\frac{1}{2}} ds \\
	&\le& \|v_{0\eps}\|_{L^\infty(\Om)}
	+ 2 c_1 R T^\frac{1}{2} \\
	&\le& R
	\qquad \mbox{for all } t\in [0,T],
  \eas
  and that
  \bas
	\big\| \big[ \Phi_3(w,v,u,\Theta)\big](\cdot,t) \big\|_{W^{1,2}(\Om)}
	&\le& \|u_{0\eps}\|_{L^\infty(\Om)}
	+ \int_0^t (t-s)^{-\frac{1}{2}} \|v\|_{W^{1,2}(\Om)} ds \\
	&\le& \|u_{0\eps}\|_{L^\infty(\Om)}
	+ RT \\
	&\le& R
	\qquad \mbox{for all } t\in [0,T],
  \eas
  while (\ref{l5}) and (\ref{l2}) in conjunction with (\ref{l9}) and (\ref{l55}) imply that
  \bas
	\big\| \big[ \Phi_4(w,v,u,\Theta)\big](\cdot,t) \big\|_{L^\infty(\Om)}
	&\le& \|\Theta_{0\eps}\|_{L^\infty(\Om)}
	+ c_2 \int_0^t (t-s)^{-\frac{1}{2}} \|\Gamma_0(\Theta) v_x^2 \|_{L^1(\Om)} ds \\
	&\le& \|\Theta_{0\eps}\|_{L^\infty(\Om)}
	+ c_2 \|\Gamma_0\|_{L^\infty([-R,R])} R^2 \int_0^t (t-s)^{-\frac{1}{2}} ds \\
	&\le& \|\Theta_{0\eps}\|_{L^\infty(\Om)}
	+ 2 c_2 \|\Gamma_0\|_{L^\infty([-R,R])} R^2 T^\frac{1}{2} \nn\\
	&\le& R
	\qquad \mbox{for all } t\in [0,T].
  \eas
  Therefore, $\Phi$ maps $S$ into itself for each $T\in (0,T_0]$, and a slight modification of this reasoning shows that if we choose
  $T=T(R)\in (0,T_0]$ suitably small, the $\Phi$ moreover acts as a contraction on $S$.
  Fixing any such $T$ now, by means of the Banach fixed point theorem we obtain $(\weps,\veps,\ueps,\Teps)\in S$
  such that $\Phi(\weps,\veps,\ueps,\Teps)=(\weps,\veps,\ueps,\Teps)$, and standard arguments from parabolic theory
  (\cite{LSU}) reveal that these functions all belong to $C^{2,1}(\bom\times [0,T]) \cap C^\infty(\bom\times (0,T])$, and that
  the problem
  \bas
	\lball
	\wepst = \eps \wepsxx + \big(\gamma_0(\Teps) \vepsx\big)_x + \big(\wh{\gamma}_0(\Teps)\uepsx\big)_x - \al\weps,
	\qquad & x\in\Om, \ t\in (0,T), \\[1mm]
	\vepst = \eps \vepsxx + \weps,
	\qquad & x\in\Om, \ t\in (0,T), \\[1mm]
	\uepst = \eps \uepsxx + \veps,
	\qquad & x\in\Om, \ t\in (0,T), \\[1mm]
	\Tepst = D \Tepsxx + \Gamma_0(\Teps) \vepsx^2,
	\qquad & x\in\Om, \ t\in (0,T), \\[1mm]
	\wepsx=\vepsx=\uepsx=\Tepsx=0,
	\qquad & x\in\pO, \ t\in (0,T), \\[1mm]
	\weps(x,0)=w_{0\eps}(x), \ \veps(x,0)=v_{0\eps}(x), \ \ueps(x,0)=u_{0\eps}(x), \ \Teps(x,0)=\Theta_{0\eps}(x),
	\quad & x\in\Om,
	\ear
  \eas
  is solved in the classical sense.
  But since $\Gamma_0(0)=\Gaeps(0)\ge 0$ and $\Theta_{0\eps} \ge 0$ in $\Om$, a comparison principle warrants that $\Teps\ge 0$
  in $\Om\times (0,T)$, and that thus we actually have $\gamma_0(\Teps)\equiv \gaeps(\Teps), \wh{\gamma}_0(\Teps)\equiv \hgeps(\Teps)$
  and $\Gamma_0(\Teps)\equiv \Gaeps(\Teps)$ in $\Om\times (0,T)$.
  Therefore, $(\weps,\veps,\ueps,\Teps)$ in fact forms a classical solution of (\ref{0eps}) in $\Om\times (0,T)$,
  so that since our choice of $T$ depends on the initial data only through the quantities
  $\|w_{0\eps}\|_{L^\infty(\Om)}, \|v_{0\eps}\|_{W^{1,2}(\Om)}, \|u_{0\eps}\|_{W^{1,2}(\Om)}$ and $\|\Theta_{0\eps}\|_{L^\infty(\Om)}$,
  a standard prolongation argument enables us to extend this quadruple to a classical solution of (\ref{0eps}) in
  $\Om\times (0,\tme)$ that satisfies (\ref{reg}) with some $\tme\in (0,\infty]$ fulfilling (\ref{ext_eps}).
  The conservations properties in (\ref{mass}), finally, result from an integration in (\ref{0eps}) using (\ref{mass0e}).
\qed
\subsection{Implications of presupposed $L^\infty$ bounds for $\Teps$ and $\Tepst$}
Now the core part of our construction of solutions will be
based on a rigorous counterpart of the formal observation that for some suitably chosen
$M>0$ one can find $B(M)>0$ and $C(M)>0$ such that the functional
\be{en}
	y(t) := \frac{1}{2} \io u_{xtt}^2
	+ \frac{1}{2} \io \gamma(\Theta) u_{xxt}^2
	+ \io \hg(\Theta) u_{xx} u_{xxt}
	+ \frac{B(M)}{2} \io u_{xx}^2
\ee
satisfies
\be{en2}
	y'(t) \le C(M)y^2(t) + C(M)y(t)
\ee
as long as
\be{en3}
	\|\Theta\|_{L^\infty(\Om)} + \|\Theta_t\|_{L^\infty(\Om)} \le M.
\ee
As a starting point of our considerations in this regard, let us record the following observation
and note that the constant $M$ appearing therein may well depend on the particular choice of the approximations
$\Gaeps, v_{0\eps}$ and $\Theta_{0\eps}$ chosen above, but that this will not affect our subsequent analysis.
\begin{lem}\label{lem01}
  There exists $M>0$ such that writing
  \be{1.2}
	\Theta_{0\eps t} := D\Theta_{0\eps xx} + \Gaeps(\Theta_{0\eps}) v_{0\eps x}^2,
	\qquad \eps\in (0,1),
  \ee
  we have
  \be{1.1}
	\|\Theta_{0\eps}\|_{L^\infty(\Om)} + \|\Theta_{0\eps t}\|_{L^\infty(\Om)} \le \frac{M}{8}
	\qquad \mbox{for all } \eps\in (0,1).
  \ee
\end{lem}
\proof
  This is a direct consequence of (\ref{gaepsc})
  and (\ref{iec}).
\qed
Let us launch our loop of arguments by fixing some constants which will play important roles in various places below.
\begin{lem}\label{lem4}
  There exist $\epss\in (0,1)$ and $k_i>0$, $i\in\{1,...,10\}$, such that if $\eps\in (0,\epss)$ and $T\in (0,\tme)$ are such that
  with $M$ as in Lemma \ref{lem01} we have
  \be{tM}
	\|\Teps(\cdot,t)\|_{L^\infty(\Om)}
	+ \|\Tepst(\cdot,t)\|_{L^\infty(\Om)} \le M
	\qquad \mbox{for all } t\in (0,T),
  \ee
  then
  \be{4.1}
	k_1 \le \gaeps(\Teps) \le k_2
	\quad \mbox{and} \quad
	k_3 \le \hgeps(\Teps) \le k_4
	\quad \mbox{and} \quad
	\Gaeps(\Teps) \le k_5
	\qquad \mbox{in } \Om\times (0,T)
  \ee
  as well as
  \be{4.2}
	|\gaeps'(\Teps)| \le k_6,
	\
	|\hgeps'(\Teps)| \le k_7,
	\
	|\gaeps''(\Teps)|\le k_8,
	\
	|\hgeps''(\Teps)| \le k_9
	\ \mbox{and} \
	|\Gaeps'(\Teps)| \le k_{10}
	\qquad \mbox{in } \Om\times (0,T).
  \ee
\end{lem}
\proof
  This immediately results from the convergence properties in (\ref{gaepsc})
  and the positivity of $\gamma$ and $\hg$ on $[0,\infty)$.
\qed
Subsequent to (\ref{4.1}), one further selection will become relevant in Lemma \ref{lem8} below:
\begin{lem}\label{lem7}
  Let $M$ and $\epss$ be as in Lemma \ref{lem01} and Lemma \ref{lem4}.
  Then there exists $k_{11}>0$ such that if $\eps\in (0,\epss)$ and $t\in (0,\tme)$ are such that (\ref{tM}) is satisfied, then
  \be{7.1}
	\|\Tepsxx(\cdot,t)\|_{L^\infty(\Om)} + \|\Tepsx(\cdot,t)\|_{L^\infty(\Om)}^2
	\le k_{11} \|\vepsx(\cdot,t)\|_{L^\infty(\Om)}^2 + k_{11}
	\qquad \mbox{for all } t\in (0,T).
  \ee
\end{lem}
\proof
  From (\ref{tM}) and (\ref{4.2}) it follows that due to the fourth equation in (\ref{0eps}),
  \bas
	|\Tepsxx| = \Big| \frac{\Tepst-\Gaeps(\Teps) \vepsx^2}{D}\Big|
	\le \frac{M}{D} + \frac{k_5}{D} \vepsx^2
	\qquad \mbox{in } \Om\times (0,T),
  \eas
  while a Gagliardo-Nirenberg interpolation yields $c_1>0$ fulfilling
  \bas
	\|\vp_x\|_{L^\infty(\Om)}^2 \le c_1\|\vp_{xx}\|_{L^\infty(\Om)} \|\vp\|_{L^\infty(\Om)}
	\qquad \mbox{for all } \vp\in C^2(\bom).
  \eas
  Therefore, (\ref{7.1}) results if we let
  $k_{11}:=(1+c_1 M) \cdot \max \big\{ \frac{M}{D},\frac{k_5}{D}\big\}$.
\qed
Now straightforward computation describes the evolution of the major part in an approximate variant of the
functional from (\ref{en}) .
\begin{lem}\label{lem3}
  Let $\eps\in (0,1)$. Then
  \bea{3.1}
	& & \hs{-20mm}
	\frac{d}{dt} \bigg\{
	\frac{1}{2} \io \wepsx^2
	+ \frac{1}{2} \io \gaeps(\Teps) \vepsxx^2
	+ \io \hgeps(\Teps) \uepsxx \vepsxx
	+ \eps \io \hgeps(\Teps) \uepsxxx^2 \bigg\} \nn\\
	& & + \al \io \wepsx^2
	+ \eps \io \wepsxx^2
	+ \eps \io \gaeps(\Teps) \vepsxxx^2
	+ 2\eps^2 \io \hgeps(\Teps) \uepsxxxx^2 \nn\\
	&=& \io \hgeps(\Teps) \vepsxx^2
	+ \frac{1}{2} \io \gaeps'(\Teps) \Tepst \vepsxx^2
	+ \io \hgeps'(\Teps) \Tepst \uepsxx \vepsxx \nn\\
	& & + \io \gaeps'(\Teps) \Tepsx \vepsxx \wepsx
	+ \io \hgeps'(\Teps) \Tepsx \uepsxx \wepsx
	+ \io \gaeps'(\Teps) \Tepsxx \vepsx \wepsx \nn\\
	& & + \io \hgeps'(\Teps) \Tepsxx \uepsx \wepsx
	+ \io \gaeps''(\Teps) \Tepsx^2 \vepsx \wepsx
	+ \io \hgeps''(\Teps) \Tepsx^2 \uepsx \wepsx \nn\\
	& &
	+ \eps \io \hgeps'(\Teps) \Tepst \uepsxxx^2
	- \eps \io \gaeps'(\Teps) \Tepsx \vepsxx \vepsxxx \nn\\
	& & - \eps \io \hgeps'(\Teps) \Tepsx \uepsxx \vepsxxx
	- \eps \io \hgeps'(\Teps) \Tepsx \uepsxxx \vepsxx \nn\\
	& & - 2\eps^2 \io \hgeps'(\Teps) \Tepsx \uepsxxx \uepsxxxx
	\qquad \mbox{for all } t\in (0,\tme).
  \eea
\end{lem}
\proof
  Testing the first equation in (\ref{0eps}) by $-\wepsxx$ shows that
  for all $t\in (0,\tme)$,
  \bea{3.2}
	\frac{1}{2} \frac{d}{dt} \io \wepsx^2
	+ \al \io \wepsx^2
	+ \eps \io \wepsxx^2
	&=& - \io \gaeps(\Teps) \vepsxx \wepsxx
	- \io \gaeps'(\Teps) \Tepsx \vepsx \wepsxx \nn\\
	& & - \io \hgeps(\Teps) \uepsxx \wepsxx
	- \io \hgeps'(\Teps) \Tepsx \uepsx \wepsxx,
  \eea
  where due to the fact that $\vepsxxt = \eps \vepsxxxx + \wepsxx$,
  \bea{3.3}
	- \io \gaeps(\Teps) \vepsxx \wepsxx
	&=& - \io \gaeps(\Teps) \vepsxx \cdot (\vepsxxt - \eps \vepsxxxx) \nn\\
	&=& - \frac{1}{2} \frac{d}{dt} \io \gaeps(\Teps) \vepsxx^2
	+ \frac{1}{2} \io \gaeps'(\Teps) \Tepst \vepsxx^2 \nn\\
	& & - \eps \io \gaeps(\Teps) \vepsxxx^2
	- \eps \io \gaeps'(\Teps) \Tepsx \vepsxx \vepsxxx
  \eea
  as well as
  \bea{3.4}
	- \io \hgeps(\Teps) \uepsxx \wepsxx
	&=& - \io \hgeps(\Teps) \uepsxx\cdot (\vepsxxt-\eps\vepsxxxx) \nn\\
	&=& - \frac{d}{dt} \io \hgeps(\Teps) \uepsxx \vepsxx
	+ \io \hgeps(\Teps) \uepsxxt \vepsxx
	+ \io \hgeps'(\Teps) \Tepst \uepsxx \vepsxx \nn\\
	& & - \eps \io \hgeps(\Teps) \uepsxxx\vepsxxx
	- \eps \io \hgeps'(\Teps) \Tepsx\uepsxx\uepsxxx
  \eea
  for all $t\in (0,\tme)$, because the boundary conditions in (\ref{0eps}) particularly imply that
  $\vepsxxx=0$ on $\pO\times (0,\tme)$.
  By the third equation in (\ref{0eps}), we similarly find that
  for all $t\in (0,\tme)$,
  \bas
	\io \hgeps(\Teps) \uepsxxt \vepsxx
	&=& \io \hgeps(\Teps) \cdot (\eps\uepsxxxx+\vepsxx)\cdot\vepsxx \nn\\
	&=& - \eps\io \hgeps(\Teps) \uepsxxx\vepsxxx
	- \eps \io \hgeps'(\Teps) \Tepsx \uepsxxx\vepsxx
	+ \io \hgeps(\Teps) \vepsxx^2,
  \eas
  and that, since also $\uepsxxx=0$ on $\pO\times (0,\tme)$,
  \bas
	-2\eps \io \hgeps(\Teps) \uepsxxx\vepsxxx
	&=& - 2\eps\io \hgeps(\Teps) \uepsxxx\cdot (\uepsxxxt - \eps\uepsxxxxx) \nn\\
	&=& - \eps \frac{d}{dt} \io \hgeps(\Teps) \uepsxxx^2
	+ \eps \io \hgeps'(\Teps) \Tepst \uepsxxx^2 \nn\\
	& & - 2\eps^2 \io \hgeps(\Teps) \uepsxxxx^2
	- 2\eps^2 \io \hgeps'(\Teps) \Tepsx \uepsxxx \uepsxxxx
  \eas
  for all $t\in (0,\tme)$, whence from (\ref{3.4}) it follows that
  \bea{3.5}
	- \io \hgeps(\Teps) \uepsxx \wepsxx
	&=& - \frac{d}{dt} \io \hgeps(\Teps) \uepsxx\vepsxx
	- \eps \frac{d}{dt} \io \hgeps(\Teps) \uepsxxx^2 \nn\\
	& & + \io \hgeps(\Teps) \vepsxx^2
	+ \io \hgeps'(\Teps) \Tepst \uepsxx\vepsxx \nn\\
	& & - \eps \io \hgeps'(\Teps) \Tepsx \uepsxx \uepsxxx
	- \eps \io \hgeps'(\Teps) \Tepsx \uepsxx\vepsxx \nn\\
	& & + \eps \io \hgeps'(\Teps) \Tepst \uepsxxx^2 \nn\\
	& & - 2\eps^2 \io \hgeps(\Teps) \uepsxxxx^2
	- 2\eps^2 \io \hgeps'(\Teps) \Tepsx \uepsxxx\uepsxxxx
  \eea
  for all $t\in (0,\tme)$.
  Moreover, the order of differentiation acting on $\weps$ in the second and last summands on the right of (\ref{3.2})
  can be reduced using two further integrations by parts according to
  \bas
	- \io \gaeps'(\Teps) \Tepsx\vepsx\wepsxx
	&=& \io \gaeps'(\Teps) \Tepsx \vepsxx\wepsx
	+ \io \gaeps'(\Teps) \Tepsxx\vepsx\wepsx
	+ \io \gaeps''(\Teps) \Tepsx^2 \vepsx \wepsx
  \eas
  and
  \bas
	\io \hgeps'(\Teps) \Tepsx\uepsx\wepsxx
	&=& \io \hgeps'(\Teps) \Tepsx\uepsxx\wepsx
	+ \io \hgeps'(\Teps) \Tepsxx\uepsx\wepsx
	+ \io \hgeps''(\Teps) \Tepsx^2 \uepsx\wepsx
  \eas
  for all $t\in (0,\tme)$.
  In conjunction with (\ref{3.5}), (\ref{3.3}) and (\ref{3.2}), this establishes (\ref{3.1}).
\qed
This information will be combined with the following elementary evolution property.
\begin{lem}\label{lem2}
  Let $\eps\in (0,1)$. Then
  \be{2.1}
	\frac{1}{2} \frac{d}{dt} \io \uepsxx^2 + \eps\io \uepsxxx^2
	\le \frac{1}{2} \io \uepsxx^2 + \frac{1}{2} \io \vepsxx^2
	\qquad \mbox{for all } t\in (0,\tme).
  \ee
\end{lem}
\proof
  This directly follows by using the third equation in (\ref{0eps}) and employing Young's inequality in confirming that
  \bas
	\frac{1}{2} \frac{d}{dt} \io \uepsxx^2
	+ \eps \io \uepsxxx^2
	= \io \uepsxx \vepsxx
	\le \frac{1}{2} \io \uepsxx^2 + \frac{1}{2} \io \vepsxx^2
  \eas
  for all $t\in (0,\tme)$.
\qed
We can now precisely define our candidate for an approximate counterpart of the functional in (\ref{en}),
and derive some basic two-sided bounds for this as long as (\ref{tM}) holds.
\begin{lem}\label{lem5}
  There exists $k_{12}>0$ such that if with $\epss$ and $(k_i)_{i\in\{1,...,10\}}$ as in Lemma \ref{lem4} we let
  \be{B}
	B:=\frac{4k_4^2}{k_1}
  \ee
  and
  \be{y}
	\yeps(t):=\frac{1}{2} \io \wepsx^2
	+ \frac{1}{2} \io \gaeps(\Teps) \vepsxx^2
	+ \io \hgeps(\Teps) \uepsxx \vepsxx
	+ \frac{B}{2} \io \uepsxx^2
	+ \eps \io \hgeps(\Teps) \uepsxxx^2
  \ee
  for $t\in (0,\tme)$ and $\eps\in (0,1)$,
  then whenever $\eps\in (0,\epss)$ and $T\in (0,\tme)$ are such that (\ref{tM}) holds with $M$ taken from Lemma \ref{lem01},
  it follows that
  \be{5.1}
	k_{12} \yeps(t)
	\ge \io \wepsx^2
	+ \io \vepsxx^2
	+ \io \uepsxx^2
	+ \eps \io \uepsxxx^2
	\qquad \mbox{for all } t\in (0,T).
  \ee
\end{lem}
\proof
  We use the lower estimates in (\ref{4.1}) to see that
  \bas
	\yeps(t)
	\ge \frac{1}{2} \io \wepsx^2
	+ \frac{k_1}{2} \io \vepsxx^2
	+ \io \hgeps(\Teps) \uepsxx\vepsxx
	+ \frac{B}{2} \io \uepsxx^2
	\qquad \mbox{for all } t\in (0,T),
  \eas
  and here a combination of Young's inequality with the rightmost bound in (\ref{4.1}) shows that
  \bas
	\bigg| \io \hgeps(\Teps) \uepsxx\vepsxx \bigg|
	\le \frac{k_1}{4} \io \vepsxx^2
	+ \frac{k_4^2}{k_1} \io \uepsxx^2
	\qquad \mbox{for all } t\in (0,T).
  \eas
  Since (\ref{B}) means that $\frac{k_4^2}{k_1}=\frac{B}{4}$, this already yields (\ref{5.1}) if we let
  $k_{12}:=\max \big\{ 2, \frac{4}{k_1}, \frac{4}{B}, \frac{1}{k_3}\big\}$.
\qed
A combination of Lemma \ref{lem3} with Lemma \ref{lem2} shows that these functionals indeed enjoy evolution properties
resembling those announced in (\ref{en2}):
\begin{lem}\label{lem8}
  With $M$ and $\epss$ taken from Lemma \ref{lem01} and Lemma \ref{lem4},
  one can find $k_{13}>0$ with the property that if $\eps\in (0,\epss)$
  and $T\in (0,\tme)$ are such that (\ref{tM}) holds, then the function in (\ref{y}) satisfies
  \be{8.1}
	\yeps'(t) + k_1 \eps \io \vepsxxx^2
	\le k_{13} \yeps^2(t) + k_{13} \yeps(t)
	\qquad \mbox{for all } t\in (0,T).
  \ee
\end{lem}
\proof
  On the right-hand side of (\ref{3.1}), we rely on our hypothesis in (\ref{tM}) in applying Lemma \ref{lem4} along with
  the Cauchy-Schwarz inequality to see that
  \be{8.2}
	\io \hgeps(\Teps) \vepsxx^2
	\le k_4 \io \vepsxx^2
	\le k_4 k_{12} \yeps(t)
  \ee
  and
  \be{8.3}
	\frac{1}{2} \io \gaeps'(\Teps) \Tepst \vepsxx^2
	\le \frac{k_6 M}{2} \io \vepsxx^2
	\le \frac{k_6 k_{12} M}{2} \yeps(t)
  \ee
  and
  \be{8.4}
	\io \hgeps'(\Teps) \Tepst \uepsxx\vepsxx
	\le k_7 M \io |\uepsxx| \cdot |\vepsxx|
	\le k_7 M \|\uepsxx\|_{L^2(\Om)} \|\vepsxx\|_{L^2(\Om)}
	\le k_7 k_{12} M \yeps(t)
  \ee
  as well as
  \be{8.5}
	\eps \io \hgeps'(\Teps) \Tepst \uepsxxx^2
	\le k_7 M \eps \io \uepsxxx^2
	\le k_7 k_{12} M \yeps(t)
  \ee
  for all $t\in (0,T)$.
  We next take $c_1>0$ such that in accordance with a Sobolev inequality we have
  \be{8.6}
	\|\vp\|_{L^2(\Om)} + \|\vp\|_{L^\infty(\Om)} \le c_1 \|\vp_x\|_{L^2(\Om)}
	\qquad \mbox{for all } \vp\in W_0^{1,2}(\Om),
  \ee
  and note that then from Lemma \ref{lem7} we particularly obtain that
  \bea{8.7}
	\|\Tepsxx\|_{L^\infty(\Om)} + \|\Tepsx\|_{L^\infty(\Om)}^2
	&\le& k_{11} c_1^2 \io \vepsxx^2 + k_{11} \nn\\
	&\le& c_1 \yeps(t) + c_2
	\qquad \mbox{for all } t\in (0,T)
  \eea
  with $c_2:=\max\{ k_{12} k_{11} c_1^2 , k_{11}\}$.
  Therefore, repeated application of Lemma \ref{lem4} together with the Cauchy-Schwarz inequality, (\ref{8.6}) and (\ref{8.7})
  shows that
  \bea{8.8}
	\io \gaeps'(\Teps) \Tepsx \vepsxx \wepsx
	&\le& k_6 \cdot (c_2\yeps(t)+c_2)^\frac{1}{2} \|\vepsxx\|_{L^2(\Om)} \|\wepsx\|_{L^2(\Om)} \nn\\
	&\le& k_6 k_{12} c_2^\frac{1}{2} (\yeps(t)+1)^\frac{1}{2} \yeps(t)
	\qquad \mbox{for all } t\in (0,T),
  \eea
  that similarly
  \bea{8.9}
	\io \hgeps'(\Teps) \Tepsx \uepsxx\wepsx
	&\le& k_7 \cdot (c_2\yeps(t)+c_2)^\frac{1}{2} \|\uepsxx\|_{L^2(\Om)} \|\wepsx\|_{L^2(\Om)} \nn\\
	&\le& k_7 k_{12} c_2^\frac{1}{2} (\yeps(t)+1)^\frac{1}{2} \yeps(t)
	\qquad \mbox{for all } t\in (0,T),
  \eea
  that
  \bea{8.10}
	\io \gaeps'(\Teps) \Tepsxx \vepsx \wepsx
	&\le& k_6\cdot (c_2\yeps(t)+c_2) \|\vepsx\|_{L^2(\Om)} \|\wepsx\|_{L^2(\Om)} \nn\\
	&\le& k_6\cdot (c_2\yeps(t)+c_2) \cdot c_1 \|\vepsxx\|_{L^2(\Om)} \|\wepsx\|_{L^2(\Om)} \nn\\
	&\le& k_6 k_{12} c_1 c_2 \yeps^2(t)
	+ k_6 c_1 c_2 \yeps(t)
	\qquad \mbox{for all } t\in (0,T)
  \eea
  and
  \be{8.11}
	\io \hgeps'(\Teps) \Tepsxx \uepsx\wepsx
	\le k_7 k_{12} c_1 c_2 \yeps^2(t) + k_7 k_{12} c_1 c_2 \yeps(t)
	\qquad \mbox{for all } t\in (0,T),
  \ee
  and that
  \bea{8.12}
	\io \gaeps''(\Teps) \Tepsx^2 \vepsx \wepsx
	&\le& k_8 \cdot (c_2\yeps(t)+c_2) \|\vepsx\|_{L^2(\Om)} \|\wepsx\|_{L^2(\Om)} \nn\\
	&\le& k_8 k_{12} c_1 c_2 \yeps^2(t) + k_8 k_{12} c_1 c_2 \yeps(t)
	\qquad \mbox{for all } t\in (0,T)
  \eea
  as well as
  \be{8.13}
	\io \hgeps''(\Teps) \Tepsx^2 \uepsx \wepsx
	\le k_9 k_{12} c_1 c_2 \yeps^2(t) + k_9 k_{12} c_1 c_2 \yeps(t)
	\qquad \mbox{for all } t\in (0,T)
  \ee
  and
  \bea{8.14}
	- \eps \io \hgeps'(\Teps) \Tepsx \uepsxxx \vepsxx
	&\le& k_7\eps\cdot (c_2\yeps(t)+c_2)^\frac{1}{2} \|\uepsxxx\|_{L^2(\Om)} \|\vepsxx\|_{L^2(\Om)} \nn\\
	&\le& k_7\eps\cdot (c_2\yeps(t)+c_2)^\frac{1}{2} \cdot \Big( \frac{k_{12} \yeps(t)}{\eps}\Big)^\frac{1}{2} \cdot
		\big( k_{12} \yeps(t)\big)^\frac{1}{2} \nn\\
	&\le& k_7 k_{12} c_2^\frac{1}{2} \cdot (\yeps(t)+1)^\frac{1}{2} \yeps(t)
	\qquad \mbox{for all } t\in (0,T),
  \eea
  because $\eps<1$.
  Apart from that, by means of Young's inequality we may draw on the dissipation-related contributions to (\ref{3.1}) in estimating
  \bea{8.15}
	& & \hs{-12mm}
	- \eps \io \gaeps(\Teps) \vepsxxx^2
	- \eps \io \gaeps'(\Teps) \Tepsx \vepsxx \vepsxxx
	- \eps \io \hgeps'(\Teps) \Tepsx \uepsxx \vepsxxx \nn\\
	&\le& - k_1 \eps \io \vepsxxx^2
	+ k_6 \eps \cdot (c_2\yeps(t)+c_2)^\frac{1}{2} \io |\vepsxx| \cdot |\vepsxxx|
	+ k_7 \eps \cdot (c_2\yeps(t)+c_2)^\frac{1}{2} \io |\uepsxx| \cdot |\vepsxxx| \nn\\
	&\le& - k_1 \eps \io \vepsxxx^2
	+ \frac{k_1\eps}{2} \io \vepsxxx^2
	+ \frac{k_6^2 \eps (c_2\yeps(t)+c_2)}{k_1} \io \vepsxx^2 \nn\\
	& & + \frac{k_1\eps}{2} \io \vepsxxx^2
	+ \frac{k_7^2 \eps (c_2\yeps(t)+c_2)}{k_1} \io \uepsxx^2 \nn\\
	&\le& \frac{(k_6^2+k_7^2) k_{12} c_2}{k_1} \yeps^2(t) + \frac{(k_6^2+k_7^2) k_{12} c_2}{k_1} \yeps(t)
	\qquad \mbox{for all } t\in (0,T)
  \eea
  and
  \bea{8.16}
	& & \hs{-20mm}
	- 2\eps^2 \io \hgeps(\Teps) \uepsxxxx^2
	- 2\eps^2 \io \hgeps'(\Teps) \Tepsx \uepsxxx \uepsxxxx \nn\\
	&\le& - 2k_3\eps^2 \io \uepsxxxx^2
	+ 2k_7\eps^2 (c_2\yeps(t)+c_2)^\frac{1}{2} \io |\uepsxxx| \cdot |\uepsxxxx| \nn\\
	&\le& \frac{k_7^2 \eps^2 (c_2\yeps(t)+c_2)}{2k_3} \io \uepsxxx^2 \nn\\
	&\le& \frac{k_7^2 k_{12} c_2}{2k_3} \yeps^2(t) + \frac{k_7^2 k_{12} c_2}{2k_3} \yeps(t)
	\qquad \mbox{for all } t\in (0,T),
  \eea
  again since $\eps<1$.
  As, finally, on the right-hand side of (\ref{2.1}) we have
  \bas
	\frac{1}{2} \io \uepsxx^2
	+ \frac{1}{2} \io \vepsxx^2
	\le \frac{k_{12}}{2} \yeps(t)
	\qquad \mbox{for all } t\in (0,T),
  \eas
  collecting (\ref{8.2})-(\ref{8.5}) and (\ref{8.8})-(\ref{8.16}) we infer from (\ref{y}), (\ref{3.1}), (\ref{2.1})
  and Lemma \ref{lem4} that
  \be{8.17}
	\yeps'(t) + k_1\eps \io \vepsxxx^2
	\le c_3 \yeps^2(t) + c_4 (\yeps(t)+1)^\frac{1}{2} \yeps(t) + c_5\yeps(t)
	\qquad \mbox{for all } t\in (0,T)
  \ee
  with $c_3:=(k_6+k_7+k_8+k_9) k_{12} c_1 c_2 + \frac{(k_6^2+k_7^2) k_{12} c_2}{k_1} + \frac{k_7^2 k_{12} c_2}{2k_3}$,
  $c_4:=(k_6+2k_7) k_{12} c_2^\frac{1}{2}$ and
  $c_5:= k_4 k_{12} + \frac{k_6 k_{12} M}{2} + 2k_7 k_{12} M + (k_6+k_7+k_8+k_9) k_{12} c_1 c_2 + \frac{(k_6^2+k_7^2) k_{12} c_2}{k_1}
  + \frac{k_7^2 k_{12} c_2}{2k_3} + \frac{k_{12}}{2}$.
  Since Young's inequality guarantees that
  \bas
	c_4(\yeps(t)+1)^\frac{1}{2} \yeps(t)
	\le \frac{c_4}{2} (\yeps(t)+2) \yeps(t)
	\qquad \mbox{for all } t\in (0,T),
  \eas
  from (\ref{8.17}) we obtain (\ref{8.1}) upon an obvious choice of $k_{13}$.
\qed
The above property will be supplemented by the following simple observation concerned with the initial instant.
\begin{lem}\label{lem9}
  Let $\epss$ and $(\yeps)_{\eps\in (0,\epss)}$ be as in Lemma \ref{lem4} and (\ref{y}).
  Then there exist $\epsss\in (0,\epss)$, $y_0>0$ and $k_{14}>0$ such that
  \be{9.1}
	\yeps(0) \le y_0
	\qquad \mbox{for all } \eps\in (0,\epsss)
  \ee
  and
  \be{9.3}
	\io v_{0\eps x}^2 \le k_{14}
	\qquad \mbox{for all } \eps\in (0,\epsss).
  \ee
\end{lem}
\proof
  Both these properties are evident consequences of (\ref{iec}) and (\ref{iec2}).
\qed
In fact, we can thereby draw from Lemma \ref{lem8} the following conclusion, yet conditional by assuming (\ref{tM}) to hold.
\begin{lem}\label{lem10}
  Let $M$ and $\epsss$ be as in Lemma \ref{lem01} and Lemma \ref{lem9}.
  Then there exist $k_{15}>0$ and $T_0>0$ such that whenever $\eps\in (0,\epsss)$
  and $T\in (0,\tme)$ are such that (\ref{tM}) holds, we have
  \be{10.1}
	\io \wepsx^2(\cdot,t) \le k_{15}
	\qquad \mbox{for all } t\in (0,T)\cap (0,T_0)
  \ee
  and
  \be{10.2}
	\io \vepsxx^2(\cdot,t) \le k_{15}
	\qquad \mbox{for all } t\in (0,T)\cap (0,T_0)
  \ee
  and
  \be{10.3}
	\io \uepsxx^2(\cdot,t) \le k_{15}
	\qquad \mbox{for all } t\in (0,T)\cap (0,T_0)
  \ee
  as well as
  \be{10.33}
	\eps \int_0^t \io \vepsxxx^2 \le k_{15}
	\qquad \mbox{for all } t\in (0,T)\cap (0,T_0).
  \ee
\end{lem}
\proof
  With $B$ and $k_{12}$ taken from Lemma \ref{lem5}, and with $y_0$ and $k_{14}$ as provided by Lemma \ref{lem9}, we let
  \be{10.4}
	N:=2(y_0+1)
  \ee
  and define
  \be{10.5}
	T_0:=\frac{1}{k_{13} N^2 + k_{13} N},
  \ee
  where $k_{13}$ is as in Lemma \ref{lem8}.
  Assuming (\ref{tM}) to hold with some $\eps\in (0,\epsss)$ and $T\in (0,\tme)$, we then infer from the continuity of the function
  $\yeps$ from (\ref{y}) that thanks to (\ref{9.1}) and (\ref{10.4}),
  \bas
	T_1:=\sup \Big\{ T'\in (0,\tme) \ \Big| \ \yeps(t) \le N \mbox{ for all } t\in (0,T') \Big\}
  \eas
  is well-defined with $T_1\in (0,\tme] \subset (0,\infty]$, and we claim that, in fact,
  \be{10.6}
	T_1 \ge \min \big\{ T \, , \, T_0\big\}.
  \ee
  To see this, we first note that in line with (\ref{tM}), from Lemma \ref{lem8} we know that
  \bea{10.8}
	\yeps'(t) + k_1 \eps \io \vepsxxx^2
	&\le& k_{13} \yeps^2(t) + k_{13} \yeps(t) \nn\\
	&\le& k_{13} N^2 + k_{13} N
	\qquad \mbox{for all } t\in (0,T)\cap (0,T_1),
  \eea
  whence assuming for contradiction that we had
  \be{10.7}
	T_1<T
	\qquad \mbox{and} \qquad
	T_1<T_0,
  \ee
  according to (\ref{10.5}) we could particularly infer from (\ref{10.8}) that
  \bas
	\yeps(t)
	&\le& \yeps(0) + (k_{13} N^2 + k_{13} N) \cdot t \\
	&\le& \yeps(0) + (k_{13} N^2 + k_{13} N) \cdot T_0 \\
	&\le& y_0+1.
  \eas
  In view of (\ref{10.4}), however, this would mean that
  \bas
	\yeps(t) \le \frac{N}{2}
	\qquad \mbox{for all } t\in (0,T_1),
  \eas
  which, again by continuity of $\yeps$, is incompatible with (\ref{10.7}) and thereby asserts (\ref{10.6}).\\
  It thus follows that $\yeps(t)\le N$ for all $t\in [0,T_1)$, so that since an integration in (\ref{10.8}) thereupon shows that
  \bas
	k_1 \eps \int_0^t \io \vepsxxx^2
	&\le& \yeps(0) + (k_{13}N^2 + k_{13} N)\cdot t \\
	&\le& N + (k_{13}N^2 + k_{13} N)\cdot T_0
	\qquad \mbox{for all } t\in (0,T)\cap (0,T_0),
  \eas
  all the estimates in (\ref{10.1})-(\ref{10.33}) have been verified if $k_{15}$ is chosen suitably large.
\qed
Now within a suitable self map type framework involving regularization properties of the Neumann heat semigroup,
the estimates in (\ref{10.1}), (\ref{10.2}) and (\ref{10.33})
can be used to make sure that the hypothesis in (\ref{tM}) is indeed satisfied in throughout a time interval
of appropriately small but $\eps$-independent length:
\begin{lem}\label{lem11}
  Let $M$, $\epss$ and $T_0$ be as provided by Lemma \ref{lem01}, Lemma \ref{lem9} and Lemma \ref{lem10}.
  Then one can find $\Ts\in (0,T_0)$ such that
  \be{11.1}
	\|\Teps(\cdot,t)\|_{L^\infty(\Om)}
	+ \|\Tepst(\cdot,t)\|_{L^\infty(\Om)}
	\le M
	\qquad \mbox{for all $t\in (0,\Ts)\cap (0,\tme)$ and } \eps\in (0,\epsss).
  \ee
\end{lem}
\proof
  According to a standard smoothing property of the Neumann heat semigroup $(e^{t\Del})_{t\ge 0}$ on $\Om$ (\cite{win_JDE2010}),
  we can fix $c_1>0$ and $c_2>0$ such that whenever $\vp\in C^0(\bom)$,
  \be{11.2}
	\|e^{tD\Del}\vp\|_{L^\infty(\Om)}
	\le c_1 t^{-\frac{1}{2}} \|\vp\|_{L^1(\Om)}
	\qquad \mbox{for all } t\in (0,T_0)
  \ee
  and
  \be{11.22}
	\|e^{tD\Del}\vp\|_{L^\infty(\Om)}
	\le c_2 t^{-\frac{1}{4}} \|\vp\|_{L^2(\Om)}
	\qquad \mbox{for all } t\in (0,T_0),
  \ee
  and by continuity of the embedding $W^{1,2}(\Om) \hra L^\infty(\Om)$ we can pick some $c_3>0$ fulfilling
  \be{11.3}
	\|\vp\|_{L^\infty(\Om)} \le c_3\|\vp_x\|_{L^2(\Om)}
	\qquad \mbox{for all } \vp\in W_0^{1,2}(\Om).
  \ee
  Taking $k_5$, $k_{10}$ and $k_{15}$ as obtained in Lemma \ref{lem4} and Lemma \ref{lem10},
  we then let $\Ts\in (0,T_0)$ be small enough such that
  \be{11.4}
	2c_1 c_3^2 k_5 k_{15} |\Om| \Ts^\frac{1}{2} \le \frac{M}{8}
  \ee
  and
  \be{11.44}
	4\sqrt{2} c_2 c_3^2 k_5 k_{15}^2 \cdot (1+T_0) \Ts^\frac{1}{4} + c_3^2 k_{10} k_{15} \Ts \le \frac{M}{8},
  \ee
  and for fixed $\eps\in (0,\epsss)$ we set
  \be{11.5}
	T_\eps:=\sup \Big\{ T'\in (0,\tme) \ \Big| \ \|\Teps(\cdot,t)\|_{L^\infty(\Om)} + \|\Tepst(\cdot,t)\|_{L^\infty(\Om)} \le M
		\mbox{ for all } t\in (0,T') \Big\},
  \ee
  noting that $T_\eps\in (0,\tme]$ is well-defined and positive by Lemma \ref{lem01} and the continuity of $\Teps$ and
  $\Tepst$ in $\bom\times [0,\tme)$.
  Now if we had
  \be{11.6}
	T_\eps < \Ts,
  \ee
  then since $T_\eps\le T_0$ by (\ref{11.4}), relying on (\ref{11.5}) we could apply Lemma \ref{lem10} to find that
  \be{11.7}
	\io \wepsx^2 \le k_{15},
	\quad
	\io \vepsxx^2 \le k_{15}
	\quad \mbox{and} \quad
	\eps \int_0^t \io \vepsxxx^2 \le k_{15}
	\qquad \mbox{for all } t\in (0,T_\eps),
  \ee
  which due to (\ref{11.3}) would particularly entail that
  \be{11.8}
	\|\vepsx\|_{L^\infty(\Om)} \le c_3 \sqrt{k_{15}}
	\qquad \mbox{for all } t\in (0,T_\eps).
  \ee
  For the functions
  \be{11.9}
	h_{1,\eps}:=\Gaeps(\Teps) \vepsx^2,
	\quad
	h_{2,\eps}:=2\Gaeps(\Teps) \vepsx \vepsxt
	\quad \mbox{and} \quad
	h_{3,\eps}:=\Gaeps'(\Teps) \vepsx^2,
  \ee
  we would therefore obtain that, by Lemma \ref{lem4} and (\ref{11.5}),
  \be{11.10}
	\|h_{1,\eps}\|_{L^1(\Om)} \le k_5 \io \vepsx^2
	\le c_3^2 k_5 k_{15} |\Om|
	\qquad \mbox{for all } t\in (0,T_\eps)
  \ee
  and
  \be{11.11}
	\|h_{3,\eps}\|_{L^\infty(\Om)}
	\le k_{10} \|\vepsx\|_{L^\infty(\Om)}^2
	\le c_3^2 k_{10} k_{15}
	\qquad \mbox{for all } t\in (0,T_\eps),
  \ee
  and that, once more thanks to (\ref{0eps}),
  \bea{11.12}
	\int_0^t \|h_{2,\eps}(\cdot,s)\|_{L^2(\Om)}^2 ds
	&\le& 2k_5 \int_0^t \|\vepsx(\cdot,s)\|_{L^\infty(\Om)}^2 \|\vepsxt(\cdot,s)\|_{L^2(\Om)}^2 ds \nn\\
	&=& 2k_5 \int_0^t \|\vepsx(\cdot,s)\|_{L^\infty(\Om)}^2 \|\eps \vepsxxx(\cdot,s)+\wepsx(\cdot,s)\|_{L^2(\Om)}^2 ds \nn\\
	&\le& 4k_5 \int_0^t \|\vepsx(\cdot,s)\|_{L^\infty(\Om)}^2 \cdot \big\{
		\eps^2 \|\vepsxxx(\cdot,s)\|_{L^2(\Om)}^2 + \|\wepsx(\cdot,s)\|_{L^2(\Om)}^2 \big\} ds \nn\\
	&\le& 4k_5 \cdot c_3^2 k_{15} \cdot (\eps k_{15} + k_{15} T_0) \nn\\
	&\le& 4c_3^2 k_5 k_{15}^2 \cdot (1+T_0)
	\qquad \mbox{for all } t\in (0,T_\eps),
  \eea
  because $\eps\le 1$.\abs
  To make appropriate use of this, we rely on Duhamel representations associated with the fourth sub-problem in (\ref{0eps})
  in verifying that since $e^{t\Del}$ is nonexpansive on $L^\infty(\Om)$ according to the maximum principle,
  Lemma \ref{lem01} along with (\ref{11.9}) and (\ref{11.4}) imply that
  \bea{11.13}
	\|\Teps(\cdot,t)\|_{L^\infty(\Om)}
	&=& \bigg\| e^{dT\Del} \Theta_{0\eps} + \int_0^t e^{D(t-s)\Del} h_{1,\eps}(\cdot,s) ds \bigg\|_{L^\infty(\Om)} \nn\\
	&\le& \|\Theta_{0\eps}\|_{L^\infty(\Om)}
	+ c_1 \int_0^t (t-s)^{-\frac{1}{2}} \|h_{1,\eps}(\cdot,s)\|_{L^1(\Om)} ds \nn\\
	&\le& \frac{M}{8}
	+ c_1 c_3^2 k_5 k_{15} |\Om| \int_0^t (t-s)^{-\frac{1}{2}} ds \nn\\
	&\le& \frac{M}{8}
	+ 2 c_1 c_3^2 k_5 k_{15} |\Om| t^\frac{1}{2} \nn\\
	&\le& \frac{M}{4}
	\qquad \mbox{for all } t\in (0,T_\eps).
  \eea
  Similarly, using that $\Tepstt=D(\Tepst)_{xx} + h_{2,\eps} + h_{3,\eps} \Tepst$ in $\Om\times (0,\tme)$ and
  $(\Tepst)_x=0$ on $\pO\times (0,\tme)$ by (\ref{0eps}) and (\ref{11.9}), on employing (\ref{11.2}) together with
  the Cauchy-Schwarz inequality, Lemma \ref{lem01}, (\ref{11.2}), (\ref{11.1}), (\ref{11.5}) and (\ref{11.44}) we see that
  \bas
	\|\Tepst(\cdot,t)\|_{L^\infty(\Om)}
	&=& \bigg\| e^{Dt\Del} \Theta_{0\eps t}
	+ \int_0^t e^{D(t-s)\Del} h_{2,\eps}(\cdot,s) ds
	+ \int_0^t e^{D(t-s)\Del} \big\{ h_{3,\eps}(\cdot,s) \Tepst(\cdot,s)\big\} ds \bigg\|_{L^\infty(\Om)} \nn\\
	&\le& \|\Theta_{0\eps t}\|_{L^\infty(\Om)}
	+ c_2 \int_0^t (t-s)^{-\frac{1}{4}} \|h_{2,\eps}(\cdot,s)\|_{L^2(\Om)} ds \nn\\
	& & + \int_0^t \|h_{3,\eps}(\cdot,s)\|_{L^\infty(\Om)} \|\Tepst(\cdot,s)\|_{L^\infty(\Om)} ds \nn\\
	&\le& \|\Theta_{0\eps t}\|_{L^\infty(\Om)}
	+ c_2\cdot\bigg\{ \int_0^t (t-s)^{-\frac{1}{2}} ds \bigg\}^\frac{1}{2} \cdot
		\bigg\{ \int_0^t \|h_{2,\eps}(\cdot,s)\|_{L^2(\Om)}^2 ds \bigg\}^\frac{1}{2} \nn\\
	& & + \int_0^t \|h_{3,\eps}(\cdot,s)\|_{L^\infty(\Om)} \|\Tepst(\cdot,s)\|_{L^\infty(\Om)} ds \nn\\
	&\le& \frac{M}{8}
	+ 4\sqrt{2} c_2 c_3^2 k_5 k_{15}^2 \cdot (1+T_0) \cdot t^\frac{1}{4}
	+ c_3^2 k_{10} k_{15} M t \nn\\
	&\le& \frac{M}{4}
	\qquad \mbox{for all } t\in (0,T_\eps).
  \eas
  In combination with (\ref{11.13}), this shows that our hypothesis would lead to the conclusion that
  \bas
	\|\Teps(\cdot,t)\|_{L^\infty(\Om)}
	+ \|\Tepst(\cdot,t)\|_{L^\infty(\Om)}
	\le \frac{M}{2}
	\qquad \mbox{for all } t\in (0,T_\eps),
  \eas
  which again by continuity of $\Teps$ and $\Tepst$ clearly is incompatible with (\ref{11.6}).
  We consequently must have $T_\eps\ge \Ts$, whence (\ref{11.1}) is entailed by the mere definition in (\ref{11.5}).
\qed
As a consequence, our approximate solutions exist at least up to the time specified above, and satisfy
estimates naturally implied by those gained from Lemma \ref{lem10} and Lemma \ref{lem11}:
\begin{cor}\label{cor12}
  Let $\epsss$ and $\Ts$ be as in Lemma \ref{lem9} and Lemma \ref{lem11}.
  Then
  \be{12.01}
	\tme \ge \Ts
	\qquad \mbox{for all } \eps\in (0,\epsss),
  \ee
  and there exists $C>0$ such that
  \be{12.1}
	\|\weps(\cdot,t)\|_{W^{1,2}(\Om)} \le C
	\qquad \mbox{for all $t\in (0,\Ts)$ and } \eps\in (0,\epsss),
  \ee
  that
  \be{12.2}
	\|\veps(\cdot,t)\|_{W^{2,2}(\Om)} \le C
	\qquad \mbox{for all $t\in (0,\Ts)$ and } \eps\in (0,\epsss),
  \ee
  that
  \be{12.3}
	\|\ueps(\cdot,t)\|_{W^{2,2}(\Om)} \le C
	\qquad \mbox{for all $t\in (0,\Ts)$ and } \eps\in (0,\epsss),
  \ee
  that
  \be{12.4}
	\|\Tepst(\cdot,t)\|_{L^\infty(\Om)} \le C
	\qquad \mbox{for all $t\in (0,\Ts)$ and } \eps\in (0,\epsss),
  \ee
  and that
  \be{12.5}
	\|\Teps(\cdot,t)\|_{W^{2,\infty}(\Om)} \le C
	\qquad \mbox{for all $t\in (0,\Ts)$ and } \eps\in (0,\epsss).
  \ee
\end{cor}
\proof
  Thanks to the continuity of the embedding $W^{1,2}(\Om) \hra L^\infty(\Om)$,
  all claimed properties immediately result from Lemma \ref{lem11} when combined with Lemma \ref{lem10},
  (\ref{mass}) and (\ref{ext_eps}).
\qed
In addition to the above, let us record some regularity properties of time derivatives.
\begin{lem}\label{lem13}
  There exists $C>0$ such that with $\epsss$ and $\Ts$ as in Lemma \ref{lem9} and Lemma \ref{lem11},
  \be{13.1}
	\|\wepst(\cdot,t)\|_{(W^{1,2}(\Om))^\star} \le C
	\qquad \mbox{for all $t\in (0,\Ts)$ and } \eps\in (0,\epsss),
  \ee
  that
  \be{13.2}
	\|\vepst(\cdot,t)\|_{L^2(\Om)} \le C
	\qquad \mbox{for all $t\in (0,\Ts)$ and } \eps\in (0,\epsss),
  \ee
  and that
  \be{13.3}
	\|\uepst(\cdot,t)\|_{L^2(\Om)} \le C
	\qquad \mbox{for all $t\in (0,\Ts)$ and } \eps\in (0,\epsss).
  \ee
\end{lem}
\proof
  For $\psi\in C^\infty(\bom)$ fulfilling $\|\psi\|_{L^2(\Om)} + \|\psi_x\|_{L^2(\Om)} \le 1$, (\ref{0eps}) implies that
  thanks to the Cauchy-Schwarz inequality,
  \bas
	\bigg| \io \wepst \psi\bigg|
	&=& \bigg| - \al \io \weps \psi
	- \eps\io \wepsx \psi_x
	- \io \gaeps(\Teps) \vepsx \psi_x
	- \io \hgeps(\Teps) \uepsx \psi_x \bigg| \\
	&\le& \al \|\weps\|_{L^2(\Om)}
	+ \eps \|\wepsx\|_{L^2(\Om)}
	+ \|\gaeps(\Teps)\|_{L^\infty(\Om)} \|\vepsx\|_{L^2(\Om)}
	+ \|\hgeps(\Teps)\|_{L^\infty(\Om)} \|\uepsx\|_{L^2(\Om)}
  \eas
  for all $t\in (0,\tme)$ and $\eps\in (0,1)$,
  whence (\ref{13.1}) results from Corollary \ref{cor12} and Lemma \ref{lem4}.\\
  Both (\ref{13.2}) and (\ref{13.3}) are immediate consequences of Corollary \ref{cor12}, because
  for all $t\in (0,\tme)$ and $\eps\in (0,1)$,
  \bas
	\|\vepst\|_{L^2(\Om)}
	\le \eps \|\vepsxx\|_{L^2(\Om)}
	+ \|\weps\|_{L^2(\Om)}
	\le \|\vepsxx\|_{L^2(\Om)}
	+ \|\weps\|_{L^2(\Om)}
  \eas
  and
  \bas
	\|\uepst\|_{L^2(\Om)}
	\le \eps \|\uepsxx\|_{L^2(\Om)}
	+ \|\veps\|_{L^2(\Om)}
	\le \|\uepsxx\|_{L^2(\Om)}
	+ \|\veps\|_{L^2(\Om)}
  \eas
  by (\ref{0eps}).
\qed
Based on the estimates collected in Corollary \ref{cor12} and Lemma \ref{lem13}, a straightforward extraction procedure now
indeed yields a strong solution of (\ref{0}) as a limit of solutions to (\ref{0eps}),
defined up to the time $\Ts$ from Lemma \ref{lem11}:
\begin{lem}\label{lem14}
  If $\epsss$ and $\Ts$ are as found in Lemma \ref{lem9} and Lemma \ref{lem11},
  then one can find $(\eps_j)_{j\in\N} \subset (0,\epsss)$ as well as functions
  \be{14.1}
	\lbal
	u \in C^0([0,\Ts];W^{2,2}_N(\Om))
	\qquad \mbox{and} \\[1mm]
	\Theta\in C^0([0,\Ts];C^1(\bom)) \cap C^{2,1}(\bom\times (0,\Ts)) \cap W^{1,\infty}(\Om\times (0,\Ts))
	\ear
  \ee
  with
  \be{14.2}
	u_t\in C^0([0,\Ts];C^1(\bom)) \cap L^\infty((0,\Ts);W^{2,2}_N(\Om))
  \ee
  and
  \be{14.3}
	u_{tt} \in C^0(\bom\times [0,\Ts]) \cap L^\infty((0\Ts);W^{1,2}(\Om))
  \ee
  such that $\Theta\ge 0$ in $\Om\times (0,\Ts)$, that $\eps_j\searrow 0$ as $j\to\infty$ and
  \begin{eqnarray}
	& & \ueps\to u
	\qquad \mbox{in } C^0([0,\Ts];C^1(\bom)),
	\label{14.4} \\
	& & \veps\to u_t
	\qquad \mbox{in } C^0([0,\Ts];C^1(\bom)),
	\label{14.5} \\
	& & \weps\to u_{tt}
	\qquad \mbox{in } C^0(\bom\times [0,\Ts]),
	\label{14.6} \\
	& & \Teps\to \Theta
	\qquad \mbox{in } C^0([0,\Ts];C^1(\bom)),
	\label{14.7} \\
	& & \Tepst\wsto \Theta_t
	\qquad \mbox{in } L^\infty(\Om\times (0,\Ts))
	\qquad \qquad \mbox{and}
	\label{14.8} \\
	& & \Tepsxx\wsto \Theta_{xx}
	\qquad \mbox{in } L^\infty(\Om\times (0,\Ts))
	\label{14.9}
  \end{eqnarray}
  as $\eps=\eps_j\searrow 0$, and that $(u,\Theta)$ forms a strong solution of (\ref{0}) in $\Om\times (0,\Ts)$
  in the sense specified in Definition \ref{dw}.
\end{lem}
\proof
  Corollary \ref{cor12} and Lemma \ref{lem13} assert that
  \bas
	(\ueps)_{\eps\in (0,\epsss)}
	\mbox{ and }
	(\veps)_{\eps\in (0,\epsss)}
	\mbox{ are bounded in } C^0([0,\Ts];W^{2,2}_N(\Om)),
  \eas
  that
  \bas
	(\weps)_{\eps\in (0,\epsss)}
	\mbox{ is bounded in } C^0([0,\Ts];W^{1,2}(\Om))
  \eas
  and
  \bas
	(\Teps)_{\eps\in (0,\epsss)}
	\mbox{ is bounded in } C^0([0,\Ts];W^{2,\infty}_N(\Om)),
  \eas
  that
  \bas
	(\uepst)_{\eps\in (0,\epsss)}
	\mbox{ and }
	(\vepst)_{\eps\in (0,\epsss)}
	\mbox{ are bounded in } L^\infty((0,\Ts);L^2(\Om)),
  \eas
  that
  \bas
	(\wepst)_{\eps\in (0,\epsss)}
	\mbox{ is bounded in } L^\infty\big( (0,\Ts); (W^{1,2}(\Om))^\star\big),
  \eas
  and that
  \bas
	(\Tepst)_{\eps\in (0,\epsss)}
	\mbox{ is bounded in } L^\infty(\Om\times (0,\Ts)).
  \eas
  By compactness of the embeddings $W^{2,2}_N(\Om) \hra C^1(\bom)$ and $W^{1,2}(\Om) \hra C^0(\bom)$, four applications
  of an Aubin-Lions lemma (\cite{simon}) thus provide $(\eps_j)_{j\in\N}\subset (0,\epsss)$ and functions
  \be{14.10}
	\lbal
	u \in C^0([0,\Ts];C^1(\bom)) \cap L^\infty((0,\Ts);W^{2,2}_N(\Om)), \\[1mm]
	v \in C^0([0,\Ts];C^1(\bom)) \cap L^\infty((0,\Ts);W^{2,2}_N(\Om)), \\[1mm]
	w\in C^0(\bom\times [0,\Ts]) \cap L^\infty((0,\Ts);W^{1,2}(\Om))
	\qquad \mbox{and} \\[1mm]
	\Theta\in C^0([0,\Ts];C^1(\bom)) \cap L^\infty((0,\Ts);W^{2,\infty}_N(\Om))
	\ear
  \ee
  such that $\eps_j\searrow 0$ as $j\to\infty$, that $\Theta_t\in L^\infty(\Om\times (0,\Ts))$,
  and that as $\eps=\eps_j\searrow 0$ we have
  \begin{eqnarray}
	& & \ueps\to u \mbox{ and } \veps\to v
	\qquad \mbox{in } C^0([0,\Ts];C^1(\bom)),
	\label{14.11} \\
	& & \ueps \wsto u \mbox{ and } \veps \wsto v
	\qquad \mbox{in } L^\infty((0,\Ts);W^{2,2}(\Om)),
	\label{14.12}, \\
	& & \weps \to w
	\qquad \mbox{in } C^0(\bom\times [0,\Ts])
	\qquad\qquad \mbox{and}
	\label{14.13} \\
	& & \weps \wsto w
	\qquad \mbox{in } L^\infty((0,\Ts);W^{1,2}(\Om))
	\label{14.14}
  \end{eqnarray}
  as well as (\ref{14.7})-(\ref{14.9}).
  Since (\ref{14.12}) and (\ref{14.14}) together with (\ref{0eps}) imply that
  \bas
	\uepst =\eps\uepsxx + \veps \wsto v
	\quad \mbox{and} \quad
	\vepst = \eps\vepsxx + \weps \wsto w
	\qquad \mbox{in } L^\infty((0,\Ts);L^2(\Om))
  \eas
  as $\eps=\eps_j\searrow 0$, it follows that, in fact, $u_t=v$ and $v_t=w$, so that not only (\ref{14.2})-(\ref{14.3})
  and hence also (\ref{14.1}), but moreover also
  (\ref{14.4})-(\ref{14.6}) result from (\ref{14.10}), (\ref{14.11}) and (\ref{14.13}).\\
  The nonnegativity of $\Theta$ is trivially inherited from that of $\Teps$ for $\eps\in (0,\epsss)$ through (\ref{14.7}),
  and (\ref{wu}) as well as the properties $u(\cdot,0)=u_0$ and $u_t(\cdot,0)=u_{0t}$ can be derived using (\ref{14.4})-(\ref{14.7})
  in a straightforward manner.
  Since (\ref{14.5}) and (\ref{14.7}) ensure that for each $\vp\in C_0^\infty(\bom\times [0,\Ts))$, in the identity
  \bas
	- \int_0^{\Ts} \io \Teps\vp_t
	- \io \Theta_{0\eps} \vp(\cdot,0)
	= - D \int_0^{\Ts} \io \Tepsx \vp_x
	+ \int_0^{\Ts} \io \Gaeps(\Teps) \vepsx^2,
	\qquad \eps\in (0,1),
  \eas
  we may let $\eps=\eps_j\searrow 0$ to verify that
  \bas
	- \int_0^{\Ts} \io \Theta \vp_t - \io \Theta_0 \vp(\cdot,0)
	= - D \int_0^{\Ts} \io \Theta_x \vp_x
	+ \int_0^{\Ts} \io \Gamma(\Theta) u_{xt}^2,
  \eas
  we furthermore obtain the problem
  \be{14.91}
	\lball
	\Theta_t = D\Theta_{xx} + h(x,t),
	\qquad & x\in\Om, \ t\in (0,\Ts), \\[1mm]
	\Theta_x=0,
	\qquad & x\in\pO, \ t\in (0,\Ts), \\[1mm]
	\Theta(x,0)=\Theta_0(x),
	\qquad & x\in\Om,
	\ear
  \ee
  is solved in the standard weak sense specified, e.g., in \cite{LSU}, where we claim that $h:=\Gamma(\Theta) u_{xt}^2$ satisfies
  \be{14.92}
	h\in C^{\vt,\frac{\vt}{2}}(\bom\times [0,\Ts])
  \ee
  with some $\vt\in (0,1)$.
  In fact,
  from (\ref{14.10}) and the identity $v=u_t$ we directly obtain that
  $h\in C^0([0,\Ts];C^0(\bom)) \cap L^\infty((0,\Ts);W^{1,2}(\Om))$,
  and fixing any $\vt_1\in (0,\frac{1}{2})$ we readily infer by means of an interpolation argument that since $W^{1,2}(\Om)$ is
  compactly embedded into $C^{\vt_1}(\bom)$, this implies that
  \be{14.93}
	h \in C^0([0,\Ts];C^{\vt_1}(\bom)).
  \ee
  Apart from that, based on the fact that in the distributional sense we have
  $h_t=\Gamma'(\Theta) \Theta_t u_{xt}^2 + 2\Gamma(\Theta) u_{xt} u_{xtt}$ we see by combining (\ref{14.1}) and (\ref{14.2})
  with (\ref{14.3}) that $h_t\in L^\infty((0,\ts);L^2(\Om))$, and that hence with some $c_1>0$,
  \be{14.94}
	\|h(\cdot,t)-h(\cdot,s)\|_{L^2(\Om)} \le c_1|t-s|
	\qquad \mbox{for all $t\in [0,\Ts]$ and } s\in [0,\Ts].
  \ee
  Taking $c_2>0$ such that in line with a Gagliardo-Nirenberg inequality we have
  \bas
	\|\vp\|_{C^0(\bom)} \le c_2 \|\vp\|_{W^{1,2}(\Om)}^\frac{1}{2} \|\vp\|_{L^2(\Om)}^\frac{1}{2}
	\qquad \mbox{for all } \vp\in W^{1,2}(\Om),
  \eas
  we can thus estimate
  \bas
	\|h(\cdot,t)-h(\cdot,s)\|_{C^0(\bom)}
	\le c_2 \cdot \big\{ 2\|h\|_{L^\infty((0,\Ts);W^{1,2}(\Om))} \big\}^\frac{1}{2} \cdot \big\{ c_1 |t-s| \big\}^\frac{1}{2}
	\qquad \mbox{for all $t\in [0,\Ts]$ and } s\in [0,\Ts],
  \eas
  meaning that in addition to (\ref{14.93}) we know that $h\in C^\frac{1}{2}([0,\Ts];C^0(\bom))$.
  As therefore (\ref{14.92}) indeed follows with some sufficiently small $\vt\in (0,1)$,
  we may return to (\ref{14.91}) and employ standard arguments from parabolic regularity theory (\cite{porzio_vespri}, \cite{LSU})
  to conclude that $\Theta$ actually
  lies in $C^{2,1}(\bom\times (0,\Ts))$ and forms a classical solution of the respective sub-problem in (\ref{0}).
\qed
\mysection{Conclusion}
Our main result, finally, can be achieved by combining Lemma \ref{lem15} and Lemma \ref{lem14} with a straightforward
extension argument as well as a simple generalization to arbitrary initial data merely satisfying (\ref{init}) without
additionally fulfilling (\ref{mass0}).\abs
\proofc of Theorem \ref{theo16}. \quad
  The uniqueness claim has been covered by Lemma \ref{lem15}, and in order to verify the statements concerning existence
  and extensibility, we first consider the case in which (\ref{mass0}) holds, that is, when the numbers
  \be{9999}
	y_0:=\frac{1}{|\Om|} \io u_0,
	\qquad
	y_1:=\frac{1}{|\Om|} \io u_{0t}
	\qquad \mbox{and} \qquad
	y_2:=\frac{1}{|\Om|} \io u_{0tt}
  \ee
  satisfy $y_0=y_1=y_2=0$.
  We may then draw on Lemma \ref{lem14} to see that according to the said uniqueness feature,
  \bas
	\tm:=\sup\Big\{ T>0 \ \Big| \ \mbox{There exists a strong solution of (\ref{0}) in $\Om\times (0,T)$} \Big\}
  \eas
  is a well-defined element of $(0,\infty]$ which has the property that there exists one single strong solution $(u,\Theta)$
  of (\ref{0}) in all of $\Om\times (0,\tm)$.\\
  Now if this solution violated (\ref{ext}), then $\tm<\infty$ and
  \be{16.1}
	\|u_t\|_{W^{2,2}(\Om)}
	+ \|u_{tt}\|_{W^{1,2}(\Om)}
	+ \|\Theta\|_{W^{2,\infty}(\Om)}
	\le c_1
	\qquad \mbox{for all } t\in (t_0,\tm)
  \ee
  with some $c_1>0$ and some $t_0\in (0,\tm)$, which particularly would imply that
  \bea{16.2}
	\|u(\cdot,t)\|_{W^{2,2}(\Om)}
	&=& \bigg\| u(\cdot,t_0) + \int_{t_0}^t u_t(\cdot,s) ds \bigg\|_{W^{2,2}(\Om)} \nn\\
	&\le& \|u(\cdot,t_0)\|_{W^{2,2}(\Om)}
	+ \int_{t_0}^{\tm} \|u_t(\cdot,s)\|_{W^{2,2}(\Om)} ds \nn\\
	&\le& \|u(\cdot,t_0)\|_{W^{2,2}(\Om)}
	+ (\tm-t_0)\cdot c_1
	\qquad \mbox{for all } t\in (t_0,\tm).
  \eea
  On the basis of (\ref{16.1}) and (\ref{16.2}), we could accordingly fix $(t_j)_{j\in\N} \subset (t_0,\tm)$ as well as
  $\wtu_0\in W^{2,2}_N(\Om)$, $\wtu_{0t} \in W^{2,2}_N(\Om)$, $\wtu_{0tt} \in W^{1,2}(\Om)$ and
  $\wtt_0\in W^{2,\infty}_N(\Om;[0,\infty))$
  such that $t_j\nearrow\tm$ as $j\to\infty$, that
  \bas
	u(\cdot,t_j) \wto \wtu_0
	\quad \mbox{and} \quad
	u_t(\cdot,t_j) \wto \wtu_{0t}
	\quad \mbox{in } W^{2,2}(\Om)
	\qquad \mbox{as } j\to\infty,
  \eas
  that
  \bas
	u_{tt}(\cdot,t_j) \wto \wtu_{0tt}
	\quad \mbox{in } W^{1,2}(\Om)
	\qquad \mbox{as } j\to\infty,
  \eas
  and that
  \bas
	\Theta(\cdot,t_j) \wsto \wtt_0
	\quad \mbox{in } W^{2,\infty}(\Om)
	\qquad \mbox{as } j\to\infty.
  \eas
  Another application of Lemma \ref{lem14}, now with $(u_0,u_{0t},u_{0tt},\Theta_0)$ and $t$
  replaced by $(\wtu_0,\wtu_{0t},\wtu_{0tt},\wtt_0)$ and $t-\tm$, would yield $T>\tm$ and a corresponding strong solution
  $(\wtu,\wtt)$ in $\Om\times (\tm,T)$ with $(\wtu,\wtu_t,\wtu_{tt},\wtt)|_{t=\tm}=(\wtu_0,\wtu_{0t},\wtu_{0tt},\wtt_0)$.
  As it can readily be confirmed that then
  \bas
	(\wh{u},\wh{\Theta})(x,t):=\lball
	(u,\Theta)(x,t)
	\qquad & \mbox{if $x\in\bom$ and } t\in [0,\tm), \\[1mm]
	(\wtu,\wtt)(x,t)
	\qquad & \mbox{if $x\in\bom$ and } t\in [\tm,T),
	\ear
  \eas
  would define a strong solution of (\ref{0}) in $\Om\times (0,T)$, this would contradict the maximality of $\tm$.\abs
  Having thus completed the proof in the case when (\ref{mass0}) holds, we are left with the general situation when
  the quantities in (\ref{9999}) are arbitrary real numbers.
  Then the functions
  $\wt{u}_0:=u_0-y_0$, $\wt{u}_{0t}:=u_{0t}-y_1$ and $\wt{u}_{0tt}:=u_{0tt}-y_2$ satisfy
  $\io \wt{u}_0 = \io \wt{u}_{0t} = \io \wt{u}_{0tt}=0$, whence
  repeating the above argument with $(u_0,u_{0t},u_{0tt})$ replaced by $(\wt{u}_0,\wt{u}_{0t},\wt{u}_{0tt})$
  we obtain $\tm\in (0,\infty]$ and a strong solution $(\wt{u},\wt{\Theta})$ of (\ref{0}) in $\Om\times (0,\tm)$
  which is such that if $\tm<\infty$, then
  $\limsup_{t\nearrow \tm} \big\{ \|\wt{u}_t(\cdot,t)\|_{W^{2,2}(\Om)} + \|\wt{u}_{tt}(\cdot,t)\|_{W^{1,2}(\Om)}
  + \|\wt{\Theta}(\cdot,t)\|_{W^{2,\infty}(\Om)} \big\} = \infty$.
  Letting $y\in C^\infty([0,\infty))$ denote the solution of
  $y'''(t)+\al y''(t)=0$
  in $(0,\infty)$ with
  $y(0)=y_0, y'(0)=y_1$ and $y''(0)=y_2$, and defining $u(x,t):=\wt{u}(x,t)+y(t)$ as well as $\Theta(x,t):=\wt{\Theta}(x,t)$
  for $(x,t)\in\bom\times [0,\tm)$, we therefore gain a pair $(u,\Theta)$ of functions which due to the fact that
  $u_x=\wt{u}_x$ and $u_{xt}=\wt{u}_{xt}$ can readily be seen to form
  a strong solution of (\ref{0}) in $\Om\times (0,\tm)$ which satisfies not only (\ref{16.01})-(\ref{16.03}) but moreover also
  (\ref{ext}).
\qed

\bigskip

{\bf Acknowledgment.} \quad
The authors acknowledge support of the Deutsche Forschungsgemeinschaft (Project No. 444955436).
They moreover declare that they have no conflict of interest.\abs
{\bf Data availability statement.} \quad
Data sharing is not applicable to this article as no datasets were
generated or analyzed during the current study.\abs

{\bf Author contribution statement.} \quad
Both authors contributed equally to the conception, execution, and writing of this manuscript.

\end{document}